\documentclass[11pt,oneside,reqno]{article}

\RequirePackage[OT1]{fontenc}
\RequirePackage{amsthm}
\usepackage[paper=a4paper]{geometry}
\usepackage[tbtags]{amsmath}
\usepackage{mathrsfs}
\usepackage{amssymb}
\usepackage{mathtools}
\usepackage{enumerate}
\usepackage{bm}
\usepackage{changes}
\usepackage{float}

\theoremstyle{plain}
\newtheorem*{thm*}{Theorem}
\newtheorem{thm}{Theorem}[section]
\newtheorem{prop}[thm]{Proposition}
\newtheorem{corr}[thm]{Corollary}

\theoremstyle{definition}
\newtheorem*{defn*}{Definition}
\newtheorem{defn}[thm]{Definition} 
\newtheorem{remark}[thm]{Remark}

\title{Convergence of Likelihood Ratios and Estimators for Selection in non-neutral Wright-Fisher Diffusions}
\date{\today}
\author{Jaromir Sant,$^{1}$, Paul A.~Jenkins,$^{2,3,4}$, Jere Koskela,$^{2}$, Dario Span{\`o},$^{2}$ \\ \\ \normalsize{MASDOC\hspace*{0.5mm}$^{1}$, Department of Statistics$^{2}$ \& Department of Computer Science$^{3}$} \\ \normalsize{University of Warwick, Coventry CV4 7AL, United Kingdom} \\ \normalsize{The Alan Turing Institute$^{4}$, British Library, London NW1 2DB, United Kingdom}}

\begin{document}

\maketitle
\begin{abstract}
	A number of discrete time, finite population size models in genetics describing the dynamics of allele frequencies are known to converge (subject to suitable scaling) to a diffusion process in the infinite population limit, termed the Wright--Fisher diffusion. This diffusion evolves on a bounded interval, so that many standard results in diffusion theory, which assume evolution on the real line, no longer apply. In this article we derive conditions to establish $\boldsymbol{\vartheta}$-uniform ergodicity for a general diffusion taking values in a bounded interval. Using these conditions, we show that the Wright--Fisher diffusion is uniformly in the selection and mutation parameters ergodic, and that the measures induced by the solution to the stochastic differential equation are uniformly locally asymptotically normal. Subsequently these two results are used to analyse the statistical properties of the Maximum Likelihood and Bayesian estimators for the selection parameter, when both selection and mutation are acting on the population. In particular, it is shown that these estimators are uniformly over compact sets consistent, display uniform in the selection parameter asymptotic normality and convergence of moments over compact sets, and are asymptotically efficient for a suitable class of loss functions.
\end{abstract}

\section{Introduction}
\noindent Mathematical population genetics is concerned with the study of how populations evolve over time, offering viable models to study how various biological phenomena such as selection and mutation affect the genetic profile of the population they act upon. Many models have been proposed over the years, but perhaps the most popular is the Wright--Fisher model (see for instance \cite[Chapter 15, Section 2]{KarlinTaylor}). \\ 
\newline
\noindent Under a suitable scaling of both space and time, a diffusion limit exists for the Wright--Fisher model, which is referred to as the Wright--Fisher diffusion, (\ref{WFDiff}), and is the main focus of this article. The Wright--Fisher diffusion is robust in the sense that a broad class of Cannings \cite{Cannings} models converge to it when suitably scaled. Furthermore, it has the neat property that the only contribution to the diffusion coefficient comes from random mating whilst other features such as selection and mutation appear solely in the drift coefficient. This facilitates inference as one can concentrate on estimating the drift, treating the diffusion coefficient as a known expression. \\
\newline
\noindent In this article we focus on a continuously observed Wright--Fisher diffusion describing the allele frequency dynamics in a two-allele, haploid population undergoing both selection and mutation. In Section 2 we start by considering a general scalar diffusion taking values in an arbitrary bounded interval $[l,r]$ (with $-\infty < l < r < \infty$) whose boundary points are either regular or entrance. In this setting, we derive verifiable criteria to establish $\boldsymbol{\vartheta}$-uniform ergodicity (see Definition \ref{UnifInVarthetaErg}), and further extend these conditions to allow for $\boldsymbol{\vartheta}$-uniform ergodicity with respect to a specific class of unbounded functions. Subsequently we introduce the Wright--Fisher diffusion, and by making use of the previously derived criteria, we show that the diffusion is ergodic uniformly in both the selection and mutation parameters, and that the associated family of measures induced by the solution to the stochastic differential equation (SDE) is uniformly locally asymptotically normal (provided the mutation parameters are greater than 1). In Section 3 we then shift our focus onto the properties of the maximum likelihood (ML) and Bayesian estimators for the selection parameter $s\in\mathcal{S}\subset\mathbb{R}$ (which measures how much more favourable one allele is over the other), under the assumption that the mutation parameters are a priori known. We briefly discuss some technical issues associated with conducting joint inference for the selection and mutation parameters in Remark \ref{InferringMutation}. \\ \newline
\noindent We point out here that by observing the path continuously through time without error, one can establish and analyse explicitly the statistical error produced by an estimator based on the whole sample path, which then sets a benchmark for the statistical performance of alternative estimators based on less informative (e.g.\ discrete) observations. In a discrete observation setting, in addition to the above mentioned statistical error, one also has to deal with observational error. One certainly cannot hope for an estimator that performs better in a discrete setting than in a continuous one, so our analysis may be viewed as the `best possible' performance for inference from a discretely observed model.\\
\newline
\noindent Inference for scalar diffusions, particularly proving consistency of estimators under specific observational schemes, has generated considerable interest over the past few years \cite{Gugushvili,Kuto,Nickl2,Nickl1,Panzar,vanderMeulen1,vanZanten1,Watterson}. However, most of the work so far has considered classes of diffusions which directly preclude the Wright--Fisher diffusion, for instance by imposing periodic boundary conditions on the drift coefficients or by requiring the diffusion coefficient to be strictly positive everywhere. The asymptotic study of a variety of estimators for continuously observed ergodic scalar diffusions has been entertained in great depth in \cite{Kuto}; see in particular Theorems 2.8 and 2.13 in \cite{Kuto}, which are respectively adaptations of Theorems I.5.1, I.10.1 and I.5.2, I.10.2 in \cite{IbraKhas}. However Theorems 2.8 and 2.13 in \cite{Kuto} cannot be applied directly to the Wright--Fisher diffusion as certain conditions do not hold, namely the reciprocal of the diffusion coefficient does not have a polynomial majorant. This discrepancy makes replicating the results for the Wright--Fisher diffusion with selection and mutation highly non-trivial. Instead we exploit the explicit nature of (\ref{WFDiff}), below, to prove, in our main result Theorem \ref{WFEstimatorResult}, uniform in the selection parameter over compact sets consistency, asymptotic normality and convergence of moments, as well as asymptotic efficiency for both the ML and Bayesian estimators. We achieve this by showing that the conditions of Theorems I.5.1, I.10.1 and I.5.2, I.10.2 in \cite{IbraKhas} still hold for the Wright--Fisher diffusion and that this diffusion is ergodic uniformly in the selection and mutation parameters (a term we define in Section \ref{WFDiffSection}). Note that the uniformity in our results for the ML and Bayesian estimators is particularly useful as it controls the lowest rate (over the true parameters) at which the parameters of interest are being learned by the inferential scheme. We further highlight that the conditions derived in Theorem \ref{UnifErgThm} provide a straightforward way to verify whether a positively recurrent diffusion on an arbitrary interval with regular or entrance boundaries is uniformly in the parameter ergodic, whilst Theorem \ref{UnifErgForUnbounded} illustrates additional conditions necessary to extend this notion for a specific class of unbounded functions when the diffusion has solely entrance boundaries.  \\
\newline
\noindent The Wright--Fisher diffusion with selection but without mutation was tackled specifically by Watterson in \cite{Watterson}, where the author makes use of a frequentist framework. Having no mutation ensures that the diffusion is absorbed at either boundary point 0 or 1 in finite time almost surely, and by conditioning on absorption Watterson computes the moment generating function, proves asymptotic normality, and derives hypothesis tests based on the Maximum Likelihood Estimator (MLE). Watterson's work however does not address the Bayesian estimator, nor does it readily extend to the case when mutation is present because the diffusion is no longer absorbed at the boundaries. In this sense the results obtained in Theorem \ref{WFEstimatorResult} are complementary to those obtained by Watterson under the assumption that the mutation parameters are known. Although this is a restriction, we are observing the path continuously over the interval $[0,T]$ and subsequently sending $T\rightarrow\infty$, so these parameters could be inferred by considering the boundary behaviour of the diffusion. More precise details about this can be found in Remark \ref{InferringMutation} in Section \ref{ResultsSection}, whilst a related argument applied to the squared Bessel process can be found in \cite[Remark 2.2]{pit:yor:1981}. \\
\newline
\noindent The rest of this article is organised as follows: in Section \ref{WFDiffSection} we start by considering a general diffusion on a bounded interval with entrance or regular boundaries, and proceed to provide verifiable criteria which ensure that the diffusion is $\boldsymbol{\vartheta}$-uniformly ergodic in Theorem \ref{UnifErgThm}. We then proceed to extend this notion to a specific class of unbounded functions for diffusions with entrance boundaries in Theorem \ref{UnifErgForUnbounded}, before moving on to introduce the Wright--Fisher diffusion and proving that this diffusion is ergodic uniformly in both the selection and mutation parameters, as well as uniformly locally asymptotically normal by making use of the general results we derive. Section \ref{ResultsSection} then focuses on the ML and Bayesian estimators for the selection parameter, proving that these estimators have a set of desirable properties in Theorem \ref{WFEstimatorResult}. The results are then supported via numerical simulations in Section \ref{Simulations}, whilst Section \ref{discussion} concludes with a discussion. The proofs of Theorems \ref{UnifErgThm} and \ref{UnifErgForUnbounded} can be found in Appendices \ref{appendix1} and \ref{appendix2} respectively.

\section{$\boldsymbol{\vartheta}$-uniform ergodicity for scalar diffusions}\label{WFDiffSection}
\noindent We start by considering an arbitrary fixed interval $[l,r]$ on which we define the SDE 
\begin{align}\label{GeneralDiffusion}
	dY_t = \mu(\boldsymbol{\vartheta},Y_t)dt + \sigma(Y_t) dW_t, & & Y_0 \sim \nu, & & \boldsymbol{\vartheta}\in\boldsymbol{\Theta}\subseteq\mathbb{R}^{d}, d \geq 1
\end{align}
where $\nu$ is an arbitrary initial distribution on $[l,r]$, $(W_{t})_{t\geq0}$ a standard Wiener process defined on a given filtered probability space, $\mu$ and $\sigma$ are such that the SDE admits a unique strong solution which we denote by $Y:= (Y_t)_{t\geq 0}$, $-\infty < l < r < \infty$ are both either entrance or regular boundaries for $Y$, and the observation interval is fixed to $[0,T]$. We denote by $\mathbb{P}_{\nu}^{(\boldsymbol{\vartheta})}$ the law induced on the space of continuous functions mapping $[0,T]$ into $[l,r]$ (henceforth denoted $C_{T}([l,r])$) by the solution to (\ref{GeneralDiffusion}) when the true diffusion parameter is set to $\boldsymbol{\vartheta}$, and $Y_0 \sim\nu$ (with dependence on $T$ being implicit). Furthermore we denote taking expectation with respect to $\mathbb{P}^{(\boldsymbol{\vartheta})}_{\nu}$ by $\mathbb{E}^{(\boldsymbol{\vartheta})}_{\nu}$. \\ \newline
\noindent Assume further that $Y$ is positive recurrent, then using standard one-dimensional diffusion theory (see Theorem 1.16 in \cite{Kuto}), we get that the unique invariant density is given by
\begin{align}\label{GeneralInvariantDensity}
	f^{Y}_{\boldsymbol{\vartheta}}(x) = \frac{1}{G^{Y}_{\boldsymbol{\vartheta}}}\frac{2}{\sigma^{2}(x)}e^{2\int^{x}\frac{\mu(\boldsymbol{\vartheta},z)}{\sigma^{2}(z)}dz}, & & x\in[l,r], & & G^{Y}_{\boldsymbol{\vartheta}} := \int_{l}^{r}\frac{2}{\sigma^{2}(x)}e^{2\int^{x}\frac{\mu(\boldsymbol{\vartheta},z)}{\sigma^{2}(z)}dz}dx.
\end{align}
\noindent In what follows, we denote taking expectation with respect to $f^{Y}_{\boldsymbol{\vartheta}}$ by $\mathbb{E}^{(\boldsymbol{\vartheta})}$, where the omission of the subscript will indicate that we start from stationarity, and henceforth always assume that $\xi\sim f^{Y}_{\boldsymbol{\vartheta}}$. \\ \newline 
\noindent In order to derive the results in Section \ref{ResultsSection}, we will need a slightly stronger notion of ergodicity which we now define. The idea here is that we can extend pointwise ergodicity in the parameter $\boldsymbol{\vartheta}$ to any compact set $\mathcal{K}\subset\Theta$ by finding the slowest rate of convergence which works within that compact set. More rigorously, we introduce the following definition. 
\begin{defn}\label{UnifInVarthetaErg}
	A process $Y$ is said to be \emph{ergodic uniformly in the parameter} $\boldsymbol{\vartheta}$ (or \emph{$\boldsymbol{\vartheta}$-uniformly ergodic}), if $\forall \varepsilon >0$ we have that
	\begin{align}\label{UnifErgDef}
		\lim\limits_{T\rightarrow \infty}\sup\limits_{\boldsymbol{\vartheta} \in\mathcal{K}}\mathbb{P}^{(\boldsymbol{\vartheta})}_{\nu}\Bigg [\left|\frac{1}{T}\int_{0}^{T}h(Y_{t})dt - \mathbb{E}^{(\boldsymbol{\vartheta})}\big [h\left(\xi\right)\big ]\right| > \varepsilon\Bigg ] = 0
	\end{align}
	holds for any $\mathcal{K}$ compact subset of the parameter space, and for any function $h:[l,r]\rightarrow\mathbb{R}$ bounded and measurable, where $\xi\sim f_{\boldsymbol{\vartheta}}^{Y}$.
\end{defn}
\noindent In the context of scalar diffusions defined on a bounded interval $[l,r]$ with $-\infty < l < r < \infty$, where both boundaries are either regular or entrance we have the following theorem:
\begin{thm}\label{UnifErgThm}
	Let $Y$ be defined as above as the solution to \eqref{GeneralDiffusion}, with boundary points $l$ and $r$ either entrance or regular, and that the expressions
	\begin{align}\label{ellConditions}
		\kappa^{l}_{\boldsymbol{\vartheta}}(a,b) &:= \int_{a}^{b}e^{-\int^{\xi}\frac{2\mu(\boldsymbol{\vartheta},y)}{\sigma^{2}(y)}dy}\int_{l}^{\xi}\frac{2}{\sigma^{2}(\eta)}e^{\int^{\eta}\frac{2\mu(\boldsymbol{\vartheta},y)}{\sigma^{2}(y)}dy}d\eta d\xi, \nonumber \\
		\kappa^{r}_{\boldsymbol{\vartheta}}(a,b) &:= \int_{a}^{b}e^{-\int^{\xi}\frac{2\mu(\boldsymbol{\vartheta},y)}{\sigma^{2}(y)}dy}\int_{\xi}^{r}\frac{2}{\sigma^{2}(\eta)}e^{\int^{\eta}\frac{2\mu(\boldsymbol{\vartheta},y)}{\sigma^{2}(y)}dy}d\eta d\xi,
	\end{align}
	are bounded away from 0 on any compact set $\mathcal{K}\subset\boldsymbol{\Theta}$, and any $l<a<b<r$. Then $Y$ is $\boldsymbol{\vartheta}$-uniformly ergodic for any initial distribution $\nu$. 
\end{thm}
\noindent We postpone the proof to Appendix \ref{appendix1}.\\ \newline
\noindent Note that the definition of $\boldsymbol{\vartheta}$-uniform ergodicity involves only bounded functions $h$, however the result above can be extended to a specific class of unbounded functions if one restricts their attention to diffusions on $[l,r]$ where $-\infty < l < r < \infty$ are both entrance boundaries.
\begin{thm}\label{UnifErgForUnbounded}
	Let $Y$ be as in Theorem \ref{UnifErgThm}, and suppose that all the conditions stated there hold, but that both $l$ and $r$ are now entrance boundaries. Assume further that the function $h$ is integrable with respect to the invariant density $f^{Y}_{\boldsymbol{\vartheta}}$ but possibly unbounded, that for any $l<a<b<r$, $\sup_{y\in[a,b]}h(y) <\infty$, and that for any $x<b$ the following hold
	\begin{align}
		&\sup_{\boldsymbol{\vartheta}\in\mathcal{K}}\int_{x}^{b}e^{-\int^{\xi}\frac{2\mu(\boldsymbol{\vartheta},y)}{\sigma^{2}(y)}dy}\int_{l}^{\xi}\frac{2h(\eta)}{\sigma^{2}(\eta)}e^{\int^{\eta}\frac{2\mu(\boldsymbol{\vartheta},y)}{\sigma^{2}(y)}dy}d\eta d\xi < \infty, \label{Unbounded1}\\
		&\sup_{\boldsymbol{\vartheta}\in\mathcal{K}}\int_{x}^{b}e^{-\int^{\xi}\frac{2\mu(\boldsymbol{\vartheta},y)}{\sigma^{2}(y)}dy}\int_{l}^{\xi}\frac{2h(\eta)}{\sigma^{2}(\eta)}e^{\int^{\eta}\frac{2\mu(\boldsymbol{\vartheta},y)}{\sigma^{2}(y)}dy}\mathbb{E}^{(\boldsymbol{\vartheta})}_{\eta}\left[\int_{0}^{T_{b}}h(Y_t)dt\right]d\eta d\xi < \infty, \label{Unbounded2}\\
		&\sup_{\boldsymbol{\vartheta}\in\mathcal{K}}\int_{l}^{r}\mathbb{E}^{(\boldsymbol{\vartheta})}_{x}\left[\int_{0}^{T_{b}}h(Y_t)dt\right]\nu(\boldsymbol{\vartheta},dx)\label{Unbounded3} < \infty,
	\end{align}
	for any compact set $\mathcal{K}\subset\boldsymbol{\Theta}$, and $T_{b} := \inf\{t\geq 0 : Y_{t} = b \}$. Then \eqref{UnifErgDef} holds for the function $h$.
\end{thm}
\noindent The proof can be found in Appendix \ref{appendix2}.
\begin{remark}
	Note that the above conditions imply that $h$ is only unbounded at the end points (because the supremum between $a$ and $b$ of $h$ is finite for any $l<a<b<r$), which in particular ensures that all integrals of the form $\int_{0}^{T}h(Y_t)dt$ above are well-defined as both boundary points are unattainable (in view of them being entrance).
\end{remark}
\subsection{The Wright--Fisher diffusion}
\noindent We now give a brief overview of the Wright--Fisher diffusion before showing that the diffusion is ergodic uniformly in the selection and mutation parameters, and subsequently use this to prove the uniform local asymptotic normality (LAN) of the family of measures associated to the solution of the SDE. \\ \newline
\noindent Consider an infinite haploid population undergoing selection and mutation, where we are interested in two alleles $A_{1}$ and $A_{2}$. Suppose that $\boldsymbol{\vartheta}=(s,\theta_1,\theta_2)\in\Theta=\mathbb{R}\times(0,\infty)^{2}$ are the selection and mutation parameters respectively, where $s$ describes the extent to which allele $A_{2}$ is favoured over $A_{1}$, alleles of type $A_{1}$ mutate to $A_{2}$ at rate $\theta_{1}/2$, and those of type $A_{2}$ mutate to $A_{1}$ at rate $\theta_{2}/2$. Let $X_{t}$ denote the frequency of $A_{2}$ in the population at time $t$. Then the dynamics of $X_{t}$ can be described by a diffusion process on $[0,1]$, which, after expressing the parameters on an appropriate timescale, satisfies the SDE
\begin{align}
	dX_{t} &= \mu_{\textnormal{WF}}(\boldsymbol{\vartheta},X_{t})dt + \sigma_{\textnormal{WF}}(X_{t})dW_{t} \nonumber\\ 
	&:= \frac{1}{2}\left(sX_{t}(1-X_{t}) - \theta_{2}X_{t} + \theta_{1}(1-X_{t})\right)dt + \sqrt{X_{t}(1-X_{t})}dW_{t} \label{WFDiff},
\end{align}	
with $X_0\sim\nu$ for some initial distribution $\nu$, and $[0,T]$ the observation interval. We point out that (\ref{WFDiff}) with $s=0$ is commonly referred to as the \emph{neutral} Wright--Fisher diffusion, whilst $s\neq0$ is known as the \emph{non-neutral} case. A strong solution to (\ref{WFDiff}) exists by the Yamada--Watanabe condition (see Theorem 3.2, Chapter IV in \cite{IkedaWatanabe}), but weak uniqueness suffices for the results in Section \ref{ResultsSection}. In abuse of notation, we redefine $\mathbb{P}^{(\boldsymbol{\vartheta})}_{\nu}$ to be the law induced on $C_{T}([0,1])$ by the solution to (\ref{WFDiff}) when the true diffusion parameters are $\boldsymbol{\vartheta}=(s,\theta_1,\theta_2)$, and $X_0 \sim\nu$, and similarly for the expectation with respect to $\mathbb{P}^{(\boldsymbol{\vartheta})}_{\nu}$, $\mathbb{E}^{(\boldsymbol{\vartheta})}_{\nu}$, and with respect to the stationary distribution, $\mathbb{E}^{(\boldsymbol{\vartheta})}$ (the existence of which we discuss below). \\ \newline
\noindent We assume that $\theta_{1}, \theta_{2} > 0$, for if at least one is 0 then the diffusion is absorbed in finite time and we are back in the regime studied by Watterson \cite{Watterson}. The boundary behaviour depends on whether the mutation parameters are either less than or greater or equal to 1, but in either case the diffusion is ergodic as long as $\theta_{1}, \theta_{2} > 0$ (see Lemma 2.1, Chapter 10 in \cite{EthierKurtzMP}). \\ \newline
\noindent Substituting $\mu_{\textnormal{WF}}$ and $\sigma_{\textnormal{WF}}$ into \eqref{GeneralInvariantDensity} and simplifying terms leads to the following density for the stationary distribution of the Wright--Fisher diffusion (\ref{WFDiff}) 
\begin{equation}\label{StationaryDensity}
	f_{\boldsymbol{\vartheta}}(x) = \frac{1}{G_{\boldsymbol{\vartheta}}}e^{sx}x^{\theta_{1}-1}(1-x)^{\theta_{2}-1}, \qquad x \in (0,1),
\end{equation}
where $G_{\boldsymbol{\vartheta}}$ is the normalising constant 
\begin{align}\label{BoundOnG}
	G_{\boldsymbol{\vartheta}} = \int_0^1 e^{sx}x^{\theta_{1}-1}(1-x)^{\theta_{2}-1}dx \leq \max\{e^{s},1\}B(\theta_{1},\theta_{2}) < \infty,
\end{align}
with
\begin{align*}
	B(\theta_{1},\theta_{2}) := \int_{0}^{1}x^{\theta_{1}-1}(1-x)^{\theta_{2}-1}dx
\end{align*}
\noindent the beta function. As above we will always assume that $\xi\sim f_{\boldsymbol{\vartheta}}$. \\ \newline 
\noindent To the best of our knowledge, it has not been shown that the Wright--Fisher diffusion is ergodic \emph{uniformly} in its parameters, which motivates the following corollary to Theorem \ref{UnifErgThm}.
\begin{corr}\label{UnifErgThm4WF}
	The Wright--Fisher diffusion with mutation and selection parameters $\boldsymbol{\vartheta}=(s,\theta_1,\theta_2)$ is $\boldsymbol{\vartheta}$-uniformly ergodic for any initial distribution $\nu$.
\end{corr}
\begin{proof}
	We show that the conditions of Theorem \ref{UnifErgThm} hold for the Wright--Fisher diffusion. Positive recurrence follows immediately from \eqref{BoundOnG}, whilst the existence of a unique strong solution is guaranteed by the Yamada--Watanabe condition. That the boundary points 0 and 1 are either entrance or regular is a consequence of the fact that the mutation parameters are assumed to be strictly positive (see (6.18) and (6.19) in \cite[Chapter 15, Section 6]{KarlinTaylor}). It remains to show that both expressions in \eqref{ellConditions} are bounded away from 0 for any $\mathcal{K}\subset\mathbb{R}\times(0,\infty)^{2}$. To this end let 
	$\bar\theta_1 := \sup_{\boldsymbol\vartheta \in \mathcal{K}} \theta_1$, $\bar\theta_2 := \sup_{\boldsymbol\vartheta \in \mathcal{K}} \theta_2$. Then
	\begin{align}
		&\int_{a}^{b}e^{-\int^{\xi}\frac{2\mu_{\textnormal{WF}}(\boldsymbol{\vartheta},y)}{\sigma^{2}_{\textnormal{WF}}(y)}dy}\int_{0}^{\xi}\frac{1}{\sigma^{2}_{\textnormal{WF}}(\eta)}e^{\int^{\eta}\frac{2\mu_{\textnormal{WF}}(\boldsymbol{\vartheta},y)}{\sigma^{2}_{\textnormal{WF}}(y)}dy}d\eta d\xi \nonumber \\
		&\qquad{}= \int_{a}^{b}2e^{-s\xi}\xi^{-\theta_{1}}(1-\xi)^{-\theta_{2}}\int_{0}^{\xi}e^{s\eta}\eta^{\theta_{1}-1}(1-\eta)^{\theta_{2}-1}d\eta d\xi \nonumber \\
		&\qquad{}\geq 2\min\{ e^{-s},1 \}\int_{a}^{b} \xi^{-\theta_{1}}(1-\xi)^{-\theta_{2}}d\xi\int_{0}^{a}\eta^{\theta_{1}-1}(1-\eta)^{\theta_{2}-1}d\eta  \nonumber \\
		&\qquad{}\geq 2\min\{e^{-s},1\}(b-a)\frac{a^{\theta_1}}{\theta_1}(1-a)^{\bar\theta_2-1}, \label{case1ell1}\\
		&\int_{a}^{b}e^{-\int^{\xi}\frac{2\mu_{\textnormal{WF}}(\boldsymbol{\vartheta},y)}{\sigma^{2}_{\textnormal{WF}}(y)}dy}\int_{\xi}^{1}\frac{1}{\sigma^{2}_{\textnormal{WF}}(\eta)}e^{\int^{\eta}\frac{2\mu_{\textnormal{WF}}(\boldsymbol{\vartheta},y)}{\sigma^{2}_{\textnormal{WF}}(y)}dy}d\eta d\xi \nonumber \\
		&\qquad{}= \int_{a}^{b}2e^{-s\xi}\xi^{-\theta_{1}}(1-\xi)^{-\theta_{2}}\int_{\xi}^{1}e^{s\eta}\eta^{\theta_{1}-1}(1-\eta)^{\theta_{2}-1}d\eta d\xi \nonumber \\
		&\qquad{}\geq 2\min\{ e^{s},1 \}\int_{a}^{b} \xi^{-\theta_{1}}(1-\xi)^{-\theta_{2}}d\xi\int_{b}^{1}\eta^{\theta_{1}-1}(1-\eta)^{\theta_{2}-1}d\eta \nonumber \\
		&\qquad{}\geq 2\min\{e^{s},1\}(b-a)\frac{(1-b)^{\theta_2}}{\theta_2}b^{\bar\theta_1-1}\label{case1ell2},
	\end{align}
	which follows by observing that
	\begin{align*}
		\xi^{-\theta_1}(1-\xi)^{-\theta_2} &> 1 & \forall \xi&\in(a,b), \forall\theta_1,\theta_2>0, \\
		(1-\eta)^{\theta_2 -1} &\geq (1-a)^{\bar{\theta}_{2}-1} & \forall\eta&\in(0,a), \\
		\eta^{\theta_1 -1} &\geq b^{\bar{\theta}_{1}-1} & \forall\eta&\in(b,1).
	\end{align*}
	As the RHS of both \eqref{case1ell1} and \eqref{case1ell2} are bounded away from 0 on $\mathcal{K}$, the result follows by applying Theorem \ref{UnifErgThm}.
\end{proof}
\noindent For the remainder of this section we restrict our attention to the parameter space $\boldsymbol{\Theta}\subset\mathbb{R}\times[1,\infty)^{2}$, where $\boldsymbol{\Theta}$ is open and bounded, for if either of the mutation parameters were less than 1 then the measures $\mathbb{P}^{(\boldsymbol{\vartheta})}_{\nu}$ within this region would be mutually singular with respect to one another and thus their Radon--Nikodym derivative undefined. Restricting our attention to mutation parameters within the range $[1,\infty)^{2}$ thus ensures that the family of measures $\{\mathbb{P}^{(\boldsymbol{\vartheta})}_{\nu},\boldsymbol{\vartheta}\in\boldsymbol{\Theta}\}$ are equivalent, and we have that 
\begin{align}\label{WFLikelihoodVartheta}
	\frac{d\mathbb{P}^{(\boldsymbol{\vartheta}')}_{\nu}}{d\mathbb{P}^{(\boldsymbol{\vartheta})}_{\nu}}(X^T) =  \frac{\nu(\boldsymbol{\vartheta}',X_0)}{\nu(\boldsymbol{\vartheta},X_0)}\exp\Bigg\{&\int_{0}^{T}\left(\frac{\mu_{\textnormal{WF}}(\boldsymbol{\vartheta}',X_t)-\mu_{\textnormal{WF}}(\boldsymbol{\vartheta},X_t)}{\sigma_{\textnormal{WF}}(X_t)}\right)dW_{t} \nonumber\\
	&\quad- \frac{1}{2}\int_{0}^{T}\left(\frac{\mu_{\textnormal{WF}}(\boldsymbol{\vartheta}',X_t)-\mu_{\textnormal{WF}}(\boldsymbol{\vartheta},X_t)}{\sigma_{\textnormal{WF}}(X_t)}\right)^{2}dt\Bigg\}
\end{align}
with $\mathbb{P}^{(\boldsymbol{\vartheta})}_{\nu}$-probability 1, where we assume the initial distributions $\{\nu(\boldsymbol{\vartheta},\cdot)\}_{\boldsymbol{\vartheta}\in\boldsymbol{\Theta}}$ are mutually equivalent and admit a density with respect to some common dominating measure $\lambda(\cdot)$, which in abuse of notation we denote by $\nu(\boldsymbol{\vartheta},\cdot)$. Proofs of the above claims regarding the equivalence of the Wright--Fisher diffusion and the form of the Radon--Nikodym derivative can be found in \cite{dawson}, Lemma 7.2.2 and Section 10.1.1. We emphasise here that we have allowed the starting distribution $\nu$ to depend on the parameters, as is evident from the first ratio in \eqref{WFLikelihoodVartheta}. However if there is no such dependence then this ratio is equal to 1 and our results still apply. \\ \newline
\noindent Furthermore, restricting to mutations greater than or equal to 1 ensures that the diffusion boundaries now become entrance (see equations (6.18) and (6.19) in \cite[Chapter 15, Section 6]{KarlinTaylor}), and as done in Theorem \ref{UnifErgForUnbounded}, \eqref{UnifErgDef} can be extended for a particular class of unbounded functions. We focus on two such functions for this class of diffusions, as they turn out to be an essential ingredient necessary to prove local asymptotic normality. 
\begin{corr}\label{UnifErgForUnboundedWF}
	For the Wright--Fisher diffusion with mutation and selection parameters $\boldsymbol{\vartheta}\in\boldsymbol{\Theta}\subset\mathbb{R}\times[1,\infty)^{2}$ (for $\boldsymbol{\Theta}$ an open bounded set) with initial distribution $\nu$ satisfying \eqref{ConditionsOnNu2}, $\boldsymbol{\vartheta}$-uniform ergodicity \eqref{UnifErgDef} holds also for the functions $h(x) = (1-x)x^{-1}$ and $h(x) = (1-x)^{-1}x$. The result holds in particular for the case $\nu = f_{\boldsymbol{\vartheta}}$.
\end{corr}
\begin{proof}
	The result follows immediately if we show that all the conditions of Theorem \ref{UnifErgForUnbounded} are satisfied for the above functions. The conditions of Theorem \ref{UnifErgThm} have already been shown to hold in Corollary \ref{UnifErgThm4WF}, whilst integrability with respect to the invariant density is guaranteed as we are considering mutation rates $(\theta_1,\theta_2)\in(1,\infty)^{2}$. That $\sup_{y\in[a,b]}h(y) <\infty$ for any pair $0<a<b<1$ is immediate, so it remains to show \eqref{Unbounded1}, \eqref{Unbounded2}, and \eqref{Unbounded3}. Observe that
	\begin{align*}
		\mathbb{E}^{(\boldsymbol{\vartheta})}_{x}\left[\int_{0}^{T_b}\frac{1-X_t}{X_t}dt\right] &= 2\int_{x}^{b}e^{-s\xi}\xi^{-\theta_1}(1-\xi)^{-\theta_2}\int_{0}^{\xi}e^{s\eta}\eta^{\theta_1 -2}(1-\eta)^{\theta_2}d\eta d\xi \nonumber\\
		&\leq 2\max\{ e^{-s},1 \}\int_{x}^{b}\xi^{-\theta_1}(1-\xi)^{-\theta_2}\int_{0}^{\xi}\eta^{\theta_1 -2}d\eta d\xi \nonumber\\
		&= 2\max\{ e^{-s},1 \}\frac{1}{\theta_1 -1}\int_{x}^{b}\xi^{-1}(1-\xi)^{-\theta_2} d\xi,
	\end{align*}
	so \eqref{Unbounded1} holds as the RHS is continuous in $\boldsymbol{\vartheta}$ and thus can be bounded from above in $\boldsymbol{\vartheta}$ over any compact set $\mathcal{K}\subset\boldsymbol{\Theta}$. For $x>b$
	\begin{align*}
		\mathbb{E}^{(\boldsymbol{\vartheta})}_{x}\left[\int_{0}^{T_b}\frac{1-X_t}{X_t}dt\right] &= 2\int_{b}^{x}e^{-s\xi}\xi^{-\theta_1}(1-\xi)^{-\theta_2}\int_{\xi}^{1}e^{s\eta}\eta^{\theta_1 -2}(1-\eta)^{\theta_2}d\eta d\xi \nonumber\\
		&\leq 2\max\{ e^{s},1 \}\int_{b}^{x}\xi^{-\max\{\theta_1,2\}}(1-\xi)^{-\theta_2}\int_{\xi}^{1}(1-\eta)^{\theta_2}d\eta d\xi\nonumber\\
		&= 2\max\{ e^{s},1 \}\frac{1}{\theta_2 +1}\int_{b}^{x}\xi^{-\max\{\theta_1,2\}}(1-\xi) d\xi\nonumber\\
		&\leq 2\max\{e^{s},1\}\frac{1}{\theta_2 + 1}\int_{b}^{x}\xi^{-\max\{\theta_1,2\}}d\xi,
	\end{align*}
	and thus \eqref{Unbounded2} holds in view of condition (\ref{ConditionsOnNu2}). In the case when $\nu=f_{\boldsymbol{\vartheta}}$, we get that
	\begin{align*}
		\mathbb{E}^{(\boldsymbol{\vartheta})}_{\nu}\left[\int_{0}^{T_b}\frac{1-X_t}{X_t}dt\right] \leq {}& 2\max\{ e^{-s},1 \}\frac{1}{\theta_1 -1}\int_{0}^{b}\int_{x}^{b}\xi^{-1}(1-\xi)^{-\theta_2} d\xi f_{\boldsymbol{\vartheta}}(x)dx\nonumber\\
		&{}+ 2\max\{ e^{s},1 \}\frac{1}{\theta_2 +1}\int_{b}^{1}\int_{b}^{x}\xi^{-\max\{\theta_1,2\}} d\xi f_{\boldsymbol{\vartheta}}(x)dx \nonumber \\
		\leq {}& 2\max\{ e^{s},1 \}\frac{1}{\theta_1 (\theta_1 -1)}\frac{1}{G_{\boldsymbol{\vartheta}}}\int_{0}^{b}(1-\xi)^{-\theta_2}d\xi\nonumber\\
		&{}+2\max\{ e^{s},1 \}\frac{1}{(\theta_2 +1)}\int_{b}^{1}\xi^{-\max\{\theta_1,2\}}d\xi,
	\end{align*}
	which follows from
	\begin{align*}
		\int_{0}^{b}\int_{x}^{b}\xi^{-1}(1-\xi)^{-\theta_2}x^{\theta_1 -1}(1-x)^{\theta_2 -1}d\xi dx &= \int_{0}^{b}\int_{0}^{\xi}\xi^{-1}(1-\xi)^{-\theta_2}x^{\theta_1 -1}(1-x)^{\theta_2 -1} dx d\xi \\
		&\leq \frac{1}{\theta_1}\int_{0}^{b}\xi^{\theta_1 -1}(1-\xi)^{-\theta_2}d\xi \\
		&\leq \frac{1}{\theta_1}\int_{0}^{b}(1-\xi)^{-\theta_2}d\xi
	\end{align*}
	because $\theta_1,\theta_2 > 1$, and
	\[
	\int_{b}^{1}\int_{b}^{x}\xi^{-\max\{\theta_1,2\}}f_{\boldsymbol{\vartheta}}(x)d\xi dx = \int_{b}^{1}\int_{\xi}^{1}\xi^{-\max\{\theta_1,2\}}f_{\boldsymbol{\vartheta}}(x) dx d\xi 
	\leq \int_{b}^{1}\xi^{-\max\{\theta_1,2\}}d\xi.
	\]
	Finally, using the recursions in (\ref{recursion1}) and (\ref{recursion2}), we get that for $x<b$,
	\begin{align*}
		\mathbb{E}^{(\boldsymbol{\vartheta})}_{x}\left[\left(\int_{0}^{T_b}\frac{1-X_t}{X_t}dt\right)^{2}\right] &\leq \frac{2(2\max\{ e^{-s},1 \})^{2}}{\theta_1 -1}\int_{0}^{b}\gamma^{\theta_1 -2}(1-\gamma)^{-\theta_2} d\gamma \nonumber \\
		&\quad\times\int_{x}^{b}\xi^{-\theta_1}(1-\xi)^{-\theta_2}d\xi \nonumber \\
		&\leq \frac{2(2\max\{ e^{-s},1 \})^{2}}{\theta_1 -1}(1-b)^{-\theta_2}\int_{0}^{b}\gamma^{\theta_1 -2} d\gamma \nonumber \\
		&\quad\times\int_{x}^{b}\xi^{-\theta_1}(1-\xi)^{-\theta_2}d\xi
	\end{align*}
	which follows from
	\begin{align*}
		\int_{0}^{\xi}\eta^{\theta_1 -2}(1-\eta)^{\theta_2}\int_{\eta}^{b}\gamma^{-1}(1-\gamma)^{-\theta_2}d\gamma d\eta &\leq \int_{0}^{b}\eta^{\theta_1 -2}(1-\eta)^{\theta_2}\int_{\eta}^{b}\gamma^{-1}(1-\gamma)^{-\theta_2}d\gamma d\eta \\
		&\leq \int_{0}^{b}\gamma^{\theta_1 -2}(1-\gamma)^{-\theta_2} d\gamma,
	\end{align*}
	and again the corresponding RHS can be bounded from above over any compact set $\mathcal{K}\subset\boldsymbol{\Theta}$ using continuity in $\boldsymbol{\vartheta}$, such that \eqref{Unbounded3} holds and so the result follows by Theorem \ref{UnifErgForUnbounded}.
\end{proof}
\noindent We end this section by introducing the concept of \emph{local asymptotic normality (LAN)} and show that the Wright--Fisher diffusion is uniformly LAN, which will be essential in the next section.
\begin{defn}[Special case of Definition 2.1 in \cite{Kuto}] The family of measures $\{\mathbb{P}^{(\boldsymbol{\vartheta})}_{\nu},\boldsymbol{\vartheta}\in\boldsymbol{\Theta}\}$ is said to be \emph{locally asymptotically normal (LAN) at a point} $\boldsymbol{\vartheta}_{0}\in\boldsymbol{\Theta}$ \emph{at rate $T^{-1/2}$} if for any $\mathbf{u}\in\mathbb{R}^{3}$, the likelihood ratio function admits the representation
	\begin{align*}
		\hspace*{-3mm}Z_{T,\boldsymbol{\vartheta}_{0}}(\mathbf{u}) &:= \frac{d\mathbb{P}^{(\boldsymbol{\vartheta}_0 + \frac{\mathbf{u}}{\sqrt{T}})}_{\nu}}{d\mathbb{P}^{(\boldsymbol{\vartheta}_0)}_{\nu}}(X^T) \\
		&\phantom{:}= \exp\left\{\left<\mathbf{u},\Delta_{T}(\boldsymbol{\vartheta}_{0},X^{T})\right> - \frac{1}{2}\left<\mathbf{I}(\boldsymbol{\vartheta}_{0})\mathbf{u},\mathbf{u}\right> + r_{T}(\boldsymbol{\vartheta}_{0},\mathbf{u},X^{T})\right\},
	\end{align*}
	where $\langle\cdot,\cdot\rangle$ denotes the Euclidean inner product on $\mathbb{R}^{3}$, and $\Delta_{T}(\boldsymbol{\vartheta}_{0},X^{T})$ is a random variable such that
	\begin{align}\label{LANNormality}
		\Delta_{T}(\boldsymbol{\vartheta}_{0},X^{T}) \stackrel{\mathclap{\normalfont\mbox{\scriptsize{d}}}}{\rightarrow} N(\mathbf{0},\mathbf{I}(\boldsymbol{\vartheta}_{0})),
	\end{align}
	with $\mathbf{I}(\boldsymbol{\vartheta}_{0})$ the Fisher information matrix evaluated at $\boldsymbol{\vartheta}_{0}$, i.e.
	\begin{align*}
		\mathbf{I}(\boldsymbol{\vartheta}_{0}) &:= \mathbb{E}^{(\boldsymbol{\vartheta}_{0})}\left[\frac{\dot{\boldsymbol{\mu}}(\boldsymbol{\vartheta}_{0},\xi)\dot{\boldsymbol{\mu}}(\boldsymbol{\vartheta}_{0},\xi)^{T}}{\sigma^{2}(\xi)}\right],
	\end{align*}
	where $\dot{\boldsymbol{\mu}}(\boldsymbol{\vartheta},\xi)^{T}$ is the transpose of the vector of derivatives of $\mu(\boldsymbol{\vartheta},x)$ with respect to $\boldsymbol{\vartheta}$. Moreover, the function $r_{T}(\boldsymbol{\vartheta}_{0},\mathbf{u},X^{T})$ satisfies
	\begin{align}\label{LANRemainder}
		\lim\limits_{T\rightarrow\infty}r_{T}(\boldsymbol{\vartheta}_{0},\mathbf{u},X^{T}) = 0 \hspace*{2mm}\textnormal{ in $\mathbb{P}^{(\boldsymbol{\vartheta}_{0})}_{\nu}$-probability}
	\end{align}
	The family of measures is said to be \emph{LAN on} $\boldsymbol{\Theta}$ if it is LAN at every point $\boldsymbol{\vartheta}_{0}\in\boldsymbol{\Theta}$, and further it is said to be \emph{uniformly LAN} on $\boldsymbol{\Theta}$ if both convergence (\ref{LANNormality}) and (\ref{LANRemainder}) are uniform in $\boldsymbol{\vartheta}\in\mathcal{K}$ for every compact $\mathcal{K}\subset\boldsymbol{\Theta}$. 
\end{defn}
\begin{thm}\label{WFDifLAN} 
	The family of measures $\{\mathbb{P}^{(\boldsymbol{\vartheta})}_{\nu},\boldsymbol{\vartheta}\in\boldsymbol{\Theta}\}$ induced by the weak solution to (\ref{WFDiff}) with initial distribution satisfying
	\begin{align}\label{ConditionsOnNu1}
		\lim_{|\boldsymbol{\varepsilon}|\rightarrow0}\frac{\nu(\boldsymbol{\vartheta}+\boldsymbol{\varepsilon},x)}{\nu(\boldsymbol{\vartheta},x)} = 1, \hspace*{3mm} \forall x\in[0,1],
	\end{align}
	\begin{align}\label{ConditionsOnNu2}
		\sup_{\boldsymbol{\vartheta}\in\mathcal{K}}\Bigg\{ \int_{0}^{b}&\frac{\max\{ e^{-s},1 \}}{\theta_1-1}\int_{x}^{b}\xi^{-1}(1-\xi)^{-\theta_2}d\xi\nu(\boldsymbol{\vartheta},dx) \nonumber \\
		&+ \int_{b}^{1}\frac{\max\{ e^{s},1 \}}{\theta_2+1}\int_{b}^{x}\xi^{-\max\{\theta_1,2\}}d\xi\nu(\boldsymbol{\vartheta},dx) \Bigg\} \leq C_{\mathcal{K}}
	\end{align}
	on any compact set $\mathcal{K}\subset\Theta$ with $C_{\mathcal{K}}>0$ constant, is uniformly LAN on $\boldsymbol{\Theta}$, with the likelihood ratio function $Z_{T,\boldsymbol{\vartheta}}(\mathbf{u})$ admitting the representation
	\begin{align*}
		Z_{T,\boldsymbol{\vartheta}}(\mathbf{u}) = \exp\left\{\left<\mathbf{u},\boldsymbol{\Delta}_{T}(\boldsymbol{\vartheta},X^{T})\right> - \frac{1}{2}\left<\mathbf{I}(\boldsymbol{\vartheta})\mathbf{u},\mathbf{u}\right> + r_{T}(\boldsymbol{\vartheta},\mathbf{u},X^{T}) \right\}
	\end{align*}
	for $\mathbf{u}\in \mathcal{U}_{T,\boldsymbol{\vartheta}}=\{\mathbf{u} : \boldsymbol{\vartheta}+\frac{\mathbf{u}}{\sqrt{T}}\in\boldsymbol{\Theta}\}$, where 
	\begin{align*}
		\boldsymbol{\Delta}_{T}(\boldsymbol{\vartheta},X^{T}) = \frac{1}{\sqrt{T}}\int_{0}^{T}\frac{\dot{\boldsymbol{\mu}}_{\textnormal{WF}}(\boldsymbol{\vartheta},X_t)}{\sigma_{\textnormal{WF}}(X_t)}dW_t. 
	\end{align*}
	In particular the result holds for $\nu = f_{\boldsymbol{\vartheta}}$.
\end{thm}
\begin{proof}
	From (\ref{WFLikelihoodVartheta}), we have that the log-likelihood ratio is given by
	\allowdisplaybreaks{\begin{align}\label{rearrange}
			\log{Z_{T,\boldsymbol{\vartheta}}(\mathbf{u})} = {}& \log\frac{\nu(\boldsymbol{\vartheta}+\frac{\mathbf{u}}{\sqrt{T}},X_0)}{\nu(\boldsymbol{\vartheta},X_0)} \nonumber \\
			&{}+ \int_{0}^{T}\frac{1}{2}\left(\frac{u_1}{\sqrt{T}}\sqrt{X_t (1-X_t)} + \frac{u_2}{\sqrt{T}}\sqrt{\frac{1-X_t}{X_t}} - \frac{u_{3}}{\sqrt{T}}\sqrt{\frac{X_t}{1-X_t}}\right)dW_t \nonumber \\
			&{}-\frac{1}{2}\int_{0}^{T}\frac{1}{4}\left(\frac{u_1}{\sqrt{T}}\sqrt{X_t (1-X_t)} + \frac{u_2}{\sqrt{T}}\sqrt{\frac{1-X_t}{X_t}} - \frac{u_{3}}{\sqrt{T}}\sqrt{\frac{X_t}{1-X_t}}\right)^{2}dt \nonumber \\
			={}& \log\frac{\nu(\boldsymbol{\vartheta}+\frac{\mathbf{u}}{\sqrt{T}},X_0)}{\nu(\boldsymbol{\vartheta},X_0)} + \left<\mathbf{u},\boldsymbol{\Delta}_{T}(\boldsymbol{\vartheta},X^T)\right> - \frac{1}{2}\left<\mathbf{I}(\boldsymbol{\vartheta})\mathbf{u},\mathbf{u}\right> \nonumber  \\
			&{}+\frac{1}{2}\left<\mathbf{I}(\boldsymbol{\vartheta})\mathbf{u},\mathbf{u}\right> - \frac{1}{2T}\int_{0}^{T}\frac{\left<\mathbf{u},\dot{\boldsymbol{\mu}}_{\textnormal{WF}}(\boldsymbol{\vartheta},X_t)\right>^{2}}{\sigma^{2}_{\textnormal{WF}}(X_t)}dt,
	\end{align}}
	where
	\begin{align*}
		\mathbf{I}(\boldsymbol{\vartheta}) = \mathbb{E}^{(\boldsymbol{\vartheta})}\left[\frac{1}{4}\begin{pmatrix}
			\xi(1-\xi) & 1-\xi & -\xi \\ 1-\xi & \frac{1-\xi}{\xi} & -1 \\ -\xi & -1 & \frac{\xi}{1-\xi}
		\end{pmatrix}\right].
	\end{align*}
	Setting
	\begin{align*}
		r_T (\boldsymbol{\vartheta},\mathbf{u},X^{T}) := \log\frac{\nu(\boldsymbol{\vartheta}+\frac{\mathbf{u}}{\sqrt{T}},X_0)}{\nu(\boldsymbol{\vartheta},X_0)} + \frac{1}{2}\left<\mathbf{I}(\boldsymbol{\vartheta})\mathbf{u},\mathbf{u}\right> - \frac{1}{2T}\int_{0}^{T}\frac{\left<\mathbf{u},\dot{\boldsymbol{\mu}}_{\textnormal{WF}}(\boldsymbol{\vartheta},X_t)\right>^{2}}{\sigma^{2}_{\textnormal{WF}}(X_t)}dt,
	\end{align*}
	we show that (\ref{LANRemainder}) holds. The first term goes to 0 as $T\to\infty$ by \eqref{ConditionsOnNu1}, and in particular $\nu = f_{\boldsymbol{\vartheta}}$ as given in (\ref{StationaryDensity}) is continuous in $\boldsymbol{\vartheta}$. Thus we deduce that (\ref{LANRemainder}) follows if we can prove that for any $\varepsilon>0$
	\begin{align}\label{UnifErgReq}
		\lim\limits_{T\rightarrow\infty}\sup_{\boldsymbol{\vartheta}\in\mathcal{K}}\mathbb{P}^{(\boldsymbol{\vartheta})}_{\nu}\left[\left|\frac{1}{T}\int_{0}^{T}\frac{\left<\mathbf{u},\dot{\boldsymbol{\mu}}_{\textnormal{WF}}(\boldsymbol{\vartheta},X_t)\right>^{2}}{\sigma^{2}_{\textnormal{WF}}(X_t)}dt -\left<\mathbf{I}(\boldsymbol{\vartheta})\mathbf{u},\mathbf{u}\right> \right|>\varepsilon\right] = 0.
	\end{align}
	Observe that the expression inside the probability in (\ref{UnifErgReq}) is made up of six distinct differences between the averages of the six distinct entries of the Fisher information matrix with respect to time and the stationary density. Thus if we are able to show that each individual difference displays the same convergence as in (\ref{UnifErgDef}), (\ref{UnifErgReq}) follows. Now, as
	\begin{align*}
		\frac{\left<\mathbf{u},\dot{\boldsymbol{\mu}}_{\textnormal{WF}}(\boldsymbol{\vartheta},x)\right>^{2}}{\sigma^{2}_{\textnormal{WF}}(x)} &= \frac{1}{4}\left( u_1\sqrt{x(1-x)} + u_2\sqrt{\frac{1-x}{x}} -u_3\sqrt{\frac{x}{1-x}} \right)^{2} \\
		&= \frac{1}{4}\left(u_{1}^{2}x(1-x) + 2u_1 u_2 (1-x) - 2u_1 u_3 x - 2u_2 u_3 + u_{2}^{2}\frac{1-x}{x} + u_{3}^{2}\frac{x}{1-x}\right)
	\end{align*}
	using (\ref{WFDiff}), we can apply Corollary \ref{UnifErgThm4WF} to the first four terms directly. The remaining two differences involve the unbounded functions $(1-x)x^{-1}$ and $x(1-x)^{-1}$, for which \eqref{UnifErgReq} has been shown to hold in Corollary \ref{UnifErgForUnboundedWF} with $\nu$ satisfying (\ref{ConditionsOnNu2}). Thus (\ref{LANRemainder}) holds (we also show in Corollary \ref{UnifErgForUnboundedWF} that (\ref{ConditionsOnNu2}) holds in the case $\nu=f_{\boldsymbol{\vartheta}}$), and (\ref{LANNormality}) follows from Proposition 1.20 in \cite{Kuto} which we can invoke in view of the above proved (\ref{UnifErgReq}) and the fact that 
	\begin{align*}
		\sup\limits_{\boldsymbol{\vartheta}\in\mathcal{K}} \sqrt{\left<\mathbf{I}(\boldsymbol{\vartheta})\mathbf{u},\mathbf{u}\right>} < \infty.
	\end{align*}
\end{proof}
\noindent We point out here that if the mutation parameters are known, condition \eqref{ConditionsOnNu2} becomes redundant and Theorem \ref{WFDifLAN} holds for any initial distribution satisfying $\lim_{\varepsilon\rightarrow0}\nu(s+\varepsilon,x)/\nu(s,x) = 1$ for any $x\in[0,1]$.

\section{Properties of the ML \& Bayesian Estimators for the Wright--Fisher diffusion}\label{ResultsSection}
We henceforth assume that the mutation parameters $\theta_1, \theta_2 >0$ are known, and thus focus on conducting inference solely on the selection parameter $s\in\mathcal{S}\subset\mathbb{R}$ with $\mathcal{S}$ open and bounded. 
\begin{remark}\label{InferringMutation}
	The continuous observation regime entertained here would enable one to infer the mutation parameters: on $\boldsymbol{\vartheta}\in\mathbb{R}\times(0,1)^{2}$ this is immediate as the family of measures $\{ \mathbb{P}_{\nu}^{(\boldsymbol{\vartheta})} : \boldsymbol{\vartheta}\in\mathbb{R}\times(0,1)^{2} \}$ are mutually singular. In particular, when either mutation parameter is less than 1, the diffusion hits the corresponding boundary in finite time almost surely, and as it does so, the diffusion coefficient (i.e.\ noise) vanishes sufficiently quickly allowing the mutation parameters to be inferred without error. Indeed, by looking at the integrands on the RHS of \eqref{rearrange}, we observe that as the path approaches either boundary, the likelihood ratio explodes. On $\boldsymbol{\vartheta}\in\mathbb{R}\times[1,\infty)^2$ the family of measures $\{ \mathbb{P}_{\nu}^{(\boldsymbol{\vartheta})} : \boldsymbol{\vartheta}\in\mathbb{R}\times[1,\infty)^2 \}$ are now mutually absolutely continuous, with both boundary points unattainable. However, the process can get arbitrarily close to either boundary as $T\to\infty$, and again the noise vanishes sufficiently quickly that the corresponding mutation parameters can be inferred to any required precision. In the case when one mutation parameter is less than 1 and the other is greater than or equal to 1, similar arguments apply. 
\end{remark}
\noindent Actually incorporating inference of the mutation parameter into the inferential setup below leads to some technical difficulties which we discuss in Section \ref{discussion}, so for simplicity we assume them to be known. Nonetheless all the notation introduced above and definitions carry through by replacing $\boldsymbol{\vartheta}$ by $s$. \\ \newline
\noindent We start by defining the MLE $\hat{s}_{T}$ of $s$ in (\ref{WFDiff}) as
\begin{align}\label{MLEDefn}
	\hat{s}_{T} = \arg\sup_{s\in\mathcal{S}}\frac{d\mathbb{P}^{(s)}_{\nu}}{d\mathbb{P}^{(s_{0})}_{\nu}}(X^T)
\end{align}
where $s_{0}\in\mathcal{S}$ is arbitrary and its only role is to specify a reference measure whose exact value does not matter. Observe that now (\ref{WFLikelihoodVartheta}) simplifies to 
\begin{align}\label{WFLikelihoodS}
	\frac{d\mathbb{P}^{(s')}_{\nu}}{d\mathbb{P}^{(s)}_{\nu}}(X^T) =  \frac{\nu(s',X_0)}{\nu(s,X_0)}\exp\Bigg\{&\int_{0}^{T}\left(s'-s\right)\sqrt{X_t(1-X_t)}dW_{t} \nonumber\\
	&\quad- \frac{1}{2}\int_{0}^{T}\left(s'-s\right)^{2}X_t(1-X_t)dt\Bigg\},
\end{align}
\noindent with initial distributions $\{\nu(s,\cdot)\}_{s\in\mathcal{S}}$ admitting a density (which we denote $\nu(s,\cdot)$) with respect to some common dominating measure $\lambda(\cdot)$. In order to be able to define the Bayesian estimator, we introduce the class $\mathscr{W}_{p}$ of loss functions $\ell : \mathcal{S} \rightarrow \mathbb{R_{+}}$ for which the following stipulations are satisfied:
\begin{enumerate}[{A}1.]
	\item $\ell(\cdot)$ is even, non-negative, and continuous at 0 with $\ell(0) = 0$ but not identically zero.
	\item The sets $\{u\in\mathcal{S}:\:\ell(u) < c \}$ are convex $\forall c>0$ (and thus $\ell(\cdot)$ is non-decreasing).
	\item $\ell(\cdot)$ has a polynomial majorant, i.e.\ there exist strictly positive constants $A$ and $b$ such that for any $u\in\mathcal{S}$,
	\begin{align*}
		|\ell(u)| \leq A(1+|u|^{b})
	\end{align*} 
	\item For any $H > 0$ sufficiently large and for sufficiently small $\gamma$, it holds that
	\begin{align*}
		\inf\limits_{|u|>H}\ell(u) - \sup\limits_{|u|\leq H^{\gamma}} \ell(u) \geq 0.
	\end{align*}
\end{enumerate}
\noindent As remarked above, we assume that $\mathcal{S}$ is an open and bounded subset of $\mathbb{R}$, and we denote by $p(\cdot)$ the prior density on $\mathcal{S}$, which we assume belongs to 
\begin{align*}
	\mathscr{P}_{c} := \left\{ p(\cdot) \in C(\bar{\mathcal{S}},\mathbb{R}_{+}) \hspace*{1mm} : \hspace*{1mm} p(u) \leq A(1+|u|^{b}) \hspace*{2mm} \forall \hspace*{1mm}u\in\bar{\mathcal{S}},\hspace*{5mm} \int_{\bar{\mathcal{S}}}p(u)du=1 \right\},
\end{align*} 
where $A$ and $b$ are some strictly positive constants, and $\bar{\mathcal{S}}$ denotes the closure of $\mathcal{S}$. 
With $p(\cdot) \in\mathscr{P}_{c}$ and $\ell(\cdot)\in\mathscr{W}_{p}$, we define the Bayesian estimator $\tilde{s}_{T}$ of $s$ in (\ref{WFDiff}) as 
\begin{align*}
	\tilde{s}_{T} = \arg\min\limits_{\bar{s}_{T}} \int_{\mathcal{S}}\mathbb{E}^{(s)}_{\nu}\left[\ell\left(\sqrt{T}\left(\bar{s}_{T}-s\right)\right)\right]p(s)ds,
\end{align*}
where the minimization is over estimators $\bar{s}_{T} = \bar{s}_{T}(X^T)$. We introduce the last class of functions we will need, namely denote by $\mathscr{G}$ the class of functions satisfying the following two conditions:
\begin{enumerate}
	\item For a fixed $T>0$, $g_{T}(\cdot)$ is a monotonically increasing function on $[0,\infty)$, with $g_{T}(y)\rightarrow\infty$ as $y\rightarrow\infty$.
	\item For any $N>0$,
	\begin{equation*}\label{functionsG1}
		\lim\limits_{\substack{T\rightarrow\infty\\ y\rightarrow\infty}}y^{N}e^{-g_{T}(y)} = 0.
	\end{equation*}
\end{enumerate} 
\noindent Observe that the likelihood ratio function is now given by
\begin{align}\label{LikelihoodRatioFnDefn}
	Z_{T,s}(u) :&= \frac{d\mathbb{P}^{(s+\frac{u}{\sqrt{T}})}_{\nu}}{d\mathbb{P}^{(s)}_{\nu}}(X^T) \nonumber \\
	&=\frac{\nu(s+\frac{u}{\sqrt{T}},X_0)}{\nu(s,X_0)}\exp\Bigg\{\left(\frac{u}{2\sqrt{T}}\right)\int_{0}^{T}\sqrt{X_t(1-X_t)}dW_{t} \nonumber\\
	&\qquad{}\qquad{}\qquad{}\qquad{}\qquad{}- \frac{1}{2}\left(\frac{u}{2\sqrt{T}}\right)^2\int_{0}^{T}X_t(1-X_t)dt\Bigg\}
\end{align}
for 
\begin{align}\label{DefnUTs}
	u\in\mathcal{U}_{T,s}:=\left\{ u\in \mathbb{R} : s + \frac{u}{\sqrt{T}} \in\mathcal{S} \right\}.
\end{align} 
\noindent We now present the main result of this article which states that the ML and Bayesian estimators for $s$ have a set of desirable properties. We prove this by showing that the conditions of Theorems I.5.1, I.5.2, I.10.1, and I.10.2 in \cite{IbraKhas} are satisfied for the Wright--Fisher diffusion. A similar formulation of the result below for the general case of a continuously observed diffusion on $\mathbb{R}$ can be found in Theorems 2.8 and 2.13 in \cite{Kuto}, where the author proves that the conditions necessary to invoke Theorems I.5.1, I.5.2, I.10.1, and I.10.2 in \cite{IbraKhas} hold for a certain class of diffusions. However, this class includes only scalar diffusions for which the inverse of the diffusion coefficient has a polynomial majorant. This fails to hold in our case, forcing us to seek alternative ways to prove that the conditions of the above mentioned theorems hold. 
\begin{thm}\label{WFEstimatorResult}
	Let $\bar{s}_{T}$ be either the ML or Bayesian estimator for the selection parameter $s\in\mathcal{S}$ (for open bounded $\mathcal{S}\subset\mathbb{R}$) in the neutral or non-neutral Wright--Fisher diffusion (\ref{WFDiff}) with initial distribution satisfying 
	\begin{align*}
		\lim_{\varepsilon\rightarrow0}\frac{\nu(s+\varepsilon,x)}{\nu(s,x)} = 1, \hspace*{3mm} \forall x \in [0,1],
	\end{align*}
	and such that for any $M\geq2$ and $u\in\mathcal{U}_{T,s}$,
	\begin{align*}
		\mathbb{P}_{\nu}^{(s)}\left[\left|\log\left(\frac{\nu(s+\frac{u}{\sqrt{T}},X_0)}{\nu(s,X_0)}\right)\right| > \frac{1}{48}\mathbb{E}^{(s)}\left[\xi(1-\xi)\right]\left|u\right|^{2}\right] \leq \frac{C_{1}}{|u|^{M}}
	\end{align*}
	and for any $R>0$ and $u,v \in \mathcal{U}_{T,s}$ with $|u| < R, |v| < R$
	\begin{align*}
		\int_{0}^{1}\left| \nu\left(s+\frac{u}{\sqrt{T}},x\right)^{\frac{1}{2}} - \nu\left(s+\frac{v}{\sqrt{T}},x\right)^{\frac{1}{2}} \right|^{2}\lambda(dx) \leq C_{2}\left|u-v\right|^{2} 
	\end{align*}
	for some constants $C_{1}, C_{2} >0$, and $\lambda(\cdot)$ common dominating measure introduced below \eqref{WFLikelihoodS} (in particular these conditions hold for the case $\nu=f_{s}$, the stationary density). Then $\bar{s}_{T}$ is uniformly over compact sets $\mathcal{K}\subset\mathcal{S}$ consistent, i.e.\ for any $\varepsilon>0$
	\begin{align*}
		\lim\limits_{T\rightarrow \infty}\sup\limits_{s\in\mathcal{K}}\mathbb{P}^{(s)}_{\nu}\big [\left|\bar{s}_{T}-s\right|>\varepsilon\big ] = 0;
	\end{align*} 
	it converges in distribution to a normal random variable
	\begin{align*}
		\sqrt{T}\left(\bar{s}_{T}-s\right) \stackrel{\mathclap{\normalfont\mbox{\scriptsize{d}}}}{\rightarrow} N(0,I(s)^{-1}),
	\end{align*}
	uniformly in $s\in\mathcal{K}$, where
	\begin{align*}
		I(s) = \frac{1}{4}\mathbb{E}^{(s)}\left[\xi(1-\xi)\right];
	\end{align*}
	and it displays moment convergence for any $p>0$
	\begin{align*}
		\lim\limits_{T\rightarrow\infty}\mathbb{E}^{(s)}_{\nu}\bigg [\left|\sqrt{T}\left(\bar{s}_{T}-s\right)\right|^{p}\bigg ] = \mathbb{E}\bigg [\left|I(s)^{-\frac{1}{2}}\zeta\right|^{p}\bigg ]
	\end{align*}
	uniformly in $s\in\mathcal{K}$, where $\zeta\sim N(0,1)$, for any compact set $\mathcal{K}\subset\mathcal{S}$. Furthermore, if the loss function $\ell(\cdot)\in\mathscr{W}_{p}$, then $\bar{s}_{T}$ is also asymptotically efficient, i.e. 
	\begin{align*}
		\lim\limits_{\delta\rightarrow 0}\lim\limits_{T\rightarrow\infty}\sup\limits_{s:|s-s_{0}|<\delta}\mathbb{E}^{(s)}_{\nu}\left[\ell\left(\sqrt{T}\left(\bar{s}_{T}-s\right)\right)\right] = \mathbb{E}\left[\ell\left(I(s_{0})^{-\frac{1}{2}}\zeta\right)\right]
	\end{align*} 
	holds for all $s_{0}\in\mathcal{S}$, where $\zeta\sim N(0,1)$.
\end{thm}
\noindent As mentioned above, the proof relies on Theorems I.5.1, I.5.2, I.10.1, and I.10.2 in \cite{IbraKhas}, which for reference we combine together in our notation into Theorem \ref{IbragimovKhasminskiiForErgDiff} below. Establishing that the conditions of Theorem \ref{IbragimovKhasminskiiForErgDiff} hold for the Wright--Fisher diffusion is non-trivial as the standard arguments found in \cite{Kuto} no longer hold, and will thus be the main focus of this section. The conclusions of Theorems I.5.1 and I.5.2 guarantee the uniform over compact sets consistency for the MLE and Bayesian estimator respectively, and also give that for any $\varepsilon>0$ and for sufficiently large $T$ 
\begin{align*}
	\sup\limits_{s\in\mathcal{K}}\mathbb{P}_{\nu}^{(s)}\left[\left|\sqrt{T}\left(\bar{s}_{T}-s\right)\right| > \varepsilon \right] \leq \alpha e^{-\beta g_{T}(\varepsilon)}
\end{align*}
with $\alpha, \beta$ strictly positive constants, and $g_{T} \in \mathscr{G}$. On the other hand, Theorems I.10.1 and I.10.2 provide the necessary conditions to deduce the uniform in $s\in\mathcal{K}$ asymptotic normality and convergence of moments for compact $\mathcal{K}\subset\mathcal{S}$, as well as asymptotic efficiency.
\begin{thm}[Ibragimov--Has'minskii]\label{IbragimovKhasminskiiForErgDiff}
	Let $\bar{s}_{T}$ denote either the ML or Bayesian estimator for the parameter $s\in\mathcal{S}$, for open bounded $\mathcal{S}\subset\mathbb{R}$, in (\ref{WFDiff}), with prior density $p(\cdot)\in\mathscr{P}_{c}$, and loss function $\ell(\cdot)\in\mathscr{W}_{p}$. Suppose further that the following conditions are satisfied by the likelihood ratio function $Z_{T,s}(u)$ as defined in (\ref{LikelihoodRatioFnDefn}):
	\begin{enumerate}
		\item $\forall\mathcal{K}\subset \mathcal{S}$ compact, we can find constants $a$ and $B$, and functions $g_{T}(\cdot)\in\mathscr{G}$ (all of which depend on $\mathcal{K}$) such that the following two conditions hold:
		\begin{itemize}
			\item $\forall R>0$, $\forall u, v \in\mathcal{U}_{T,s}$ as defined in (\ref{DefnUTs}) satisfying $|u|<R$, $|v|<R$, and for some $m\geq q>1$
			\begin{align}\label{IbragimovLikeRatioBoundsForErgDiff}
				\sup\limits_{s\in\mathcal{K}}\mathbb{E}^{(s)}_{\nu}\left[\left|Z_{T,s}(u)^{\frac{1}{m}} - Z_{T,s}(v)^{\frac{1}{m}}\right|^{m}\right] \leq B(1+R^{a})|u-v|^{q}.
			\end{align}
			\item $\forall u\in \mathcal{U}_{T,s}$
			\begin{align*}
				\sup\limits_{s\in\mathcal{K}}\mathbb{E}^{(s)}_{\nu}\left[Z_{T,s}(u)^{\frac{1}{2}}\right] \leq e^{-g_{T}(|u|)}.
			\end{align*}
		\end{itemize}
		\item The random functions $Z_{T,s}(u)$ have marginal distributions which converge uniformly in $s\in\mathcal{K}$ as $T\rightarrow\infty$ to those of the random function $Z_{s}(u)\in C_{0}(\mathbb{R})$, where $C_{0}(\mathbb{R})$ denotes the space of continuous functions on $\mathbb{R}$ vanishing at infinity, equipped with the supremum norm and the Borel $\sigma$-algebra.
		\item The limit function $Z_{s}(u)$ attains its maximum at the unique point $\hat{u}(s)=u$ with probability 1, and the random function
		\begin{align*}
			\psi(v) = \int_{\mathbb{R}}\ell(v-u)\frac{Z_{s}(u)}{\int_{\mathbb{R}}Z_{s}(y)dy}du
		\end{align*}
		attains its minimum value at a unique point $\tilde{u}(s) = u$ with probability 1.
	\end{enumerate}
	Then we have that $\bar{s}_{T}$ is: uniformly in $s\in\mathcal{K}$ consistent, i.e.\ for any $\varepsilon>0$
	\begin{align*}
		\lim\limits_{T\rightarrow \infty}\sup\limits_{s\in\mathcal{K}}\mathbb{P}^{(s)}_{\nu}\big [\left|\bar{s}_{T}-s\right|>\varepsilon\big ] = 0,
	\end{align*} 
	the distributions of the random variables $\bar{u}_{T}=\sqrt{T}\left(\bar{s}_{T}-s\right)$ converge uniformly in $s\in\mathcal{K}$ to the distribution of $\bar{u}$, and for any loss function $\ell\in\mathscr{W}_{p}$ uniformly in $s\in\mathcal{K}$
	\begin{align}\label{UnifTildeU}
		\lim\limits_{T\rightarrow\infty}\mathbb{E}^{(s)}_{\nu}\left[\ell\left(\sqrt{T}\left(\bar{s}_{T}-s\right)\right)\right] = \mathbb{E}^{(s)}_{\nu}\left[\ell(\bar{u})\right].
	\end{align}
\end{thm}
\noindent For the Bayesian estimator, the requirements for inequality (\ref{IbragimovLikeRatioBoundsForErgDiff}) can be weakened as it suffices to show that (\ref{IbragimovLikeRatioBoundsForErgDiff}) holds for $m=2$ and any $q>0$.

\begin{proof}[Proof of Theorem \ref{WFEstimatorResult}]
	Our aim will be to prove that Conditions 1, 2, and 3 in Theorem \ref{IbragimovKhasminskiiForErgDiff} hold for the Wright--Fisher diffusion, for then the ML and Bayesian estimator are uniformly on compact sets consistent. Below, Condition 1 is shown to hold in Propositions \ref{Inequality1} and \ref{Lemma2.11AltResult}; Condition 2 is shown in Corollary \ref{MarginalConvergence}; and Condition 3 is shown in Proposition \ref{thm4.6}.\\ \newline
	\noindent It remains to show how uniform in $s\in\mathcal{K}$ asymptotic normality and convergence of moments, as well as asymptotic efficiency (under the right choice of loss function) follow. Given Conditions 1, 2, and 3 of Theorem 3.2, uniform in $s\in\mathcal{K}$ asymptotic normality follows immediately from Proposition \ref{thm4.6}; $\bar{u} = I(s)^{-1}\Delta(s)$, $\Delta(s) \sim N(0,I(s))$, and $\bar{u}_{T}$ converges uniformly in distribution to $\bar{u}$. Moreover, as stated in Remark I.5.1 in \cite{IbraKhas}, the Ibragimov--Has'minskii conditions also give us a bound on the tails of the likelihood ratio, which can be translated into bounds on the tails of $|\hat{u}_{T}|^{p}$ for any $p>0$ (see the display just below (2.27) in \cite{Kuto}). Similar bounds on the tails of $|\tilde{u}_{T}|^{p}$ hold for the Bayesian estimator by Theorem I.5.7 in \cite{IbraKhas}, and thus we have that the random variables $|\bar{u}_{T}|^{p}$ are uniformly integrable for any $p>0$, uniformly in $s\in\mathcal{K}$ for any compact $\mathcal{K}\subset\mathcal{S}$. Uniform convergence of the moments of the estimators follows from this and the uniform convergence in distribution (by applying a truncation argument). \\ \newline
	For loss functions satisfying $\ell(\cdot)\in\mathscr{W}_{p}$, observe that the uniform convergence in (\ref{UnifTildeU}) allows us to deduce that
	\begin{align*}
		\lim\limits_{T\rightarrow\infty}\sup\limits_{s:|s-s_{0}|<\delta}\mathbb{E}^{(s)}_{\nu}\left[\ell\left(\sqrt{T}\left(\bar{s}_{T}-s\right)\right)\right] = \sup\limits_{s:|s-s_{0}|<\delta}\mathbb{E}\left[\ell\left(I(s)^{-\frac{1}{2}}\zeta\right)\right]
	\end{align*}
	for $\zeta\sim N(0,1)$. As $I(s)$ is continuous in $s$, we have that
	\begin{align*}
		\lim\limits_{\delta\rightarrow 0}\sup\limits_{s:|s-s_{0}|<\delta}\mathbb{E}\left[\ell\left(I(s)^{-\frac{1}{2}}\zeta\right)\right] = \mathbb{E}\left[\ell\left(I(s_{0})^{-\frac{1}{2}}\zeta\right)\right],
	\end{align*}
	giving asymptotic efficiency.
\end{proof}
\noindent We proceed to show that Conditions 1, 2, and 3 in Theorem \ref{IbragimovKhasminskiiForErgDiff} hold for the Wright--Fisher diffusion. Theorem \ref{WFDifLAN} gives us that the Wright--Fisher diffusion is uniformly LAN, which immediately gives the required marginal convergence of the $Z_{T,s}(u)$ in Condition 2.
\begin{corr}\label{MarginalConvergence}
	For any initial distribution satisfying 
	\begin{align*}
		\lim_{\varepsilon\rightarrow0}\frac{\nu(s+\varepsilon,x)}{\nu(s,x)} = 1 \hspace*{3mm} \forall x \in [0,1],
	\end{align*}
	the random functions $Z_{T,s}(u)$ given by 
	\begin{align*}
		Z_{T,s}(u) &= \exp\left\{\frac{u}{2\sqrt{T}}\int_{0}^{T}\sqrt{X_{t}(1-X_{t})}dW_{t} - \frac{u^{2}}{8}\mathbb{E}^{(s)}\left[\xi(1-\xi)\right] + r_{T}(s,u,X^{T}) \right\} \\
		&=: \exp\left\{u\Delta_{T}(s) - \frac{u^{2}}{2}I(s) + r_{T}(s,u,X^{T})\right\},
	\end{align*}
	where
	\begin{align*}
		r_{T}(s,u,X^{T}) := \log\left(\frac{\nu(s+\frac{u}{\sqrt{T}},X_0)}{\nu(s,X_0)}\right)+\frac{u^2}{8}\mathbb{E}^{(s)}\left[\xi(1-\xi)\right] - \frac{1}{2}\left(\frac{u}{2\sqrt{T}}\right)^{2}\int_{0}^{T}X_t(1-X_t)dt,
	\end{align*}
	have marginal distributions which converge uniformly in $s\in\mathcal{K}$ as $T\rightarrow\infty$ to those of the random function $Z_{s}(u)\in C_{0}(\mathbb{R})$ given by
	\begin{align*}
		Z_{s}(u) &:= \exp\left\{u\Delta(s) - \frac{u^{2}}{2}I(s)\right\},
	\end{align*}
	where 
	\begin{align*}
		\Delta(s) := \lim\limits_{T\rightarrow\infty}\frac{1}{2\sqrt{T}}\int_{0}^{T}\sqrt{X_{t}(1-X_{t})}dW_{t} \sim N\left(0,I(s)\right).
	\end{align*}
\end{corr}
\begin{proof}
	The result follows immediately from the uniform LAN of the family of measures as shown in Theorem \ref{WFDifLAN}; see for illustration the display just before Lemma 2.10 in \cite{Kuto}. It is clear that $Z_{s}(u)$ vanishes at infinity and thus is an element of $C_{0}(\mathbb{R})$.
\end{proof}
\noindent The next two results allow us to control the Hellinger distance of the likelihood ratio function as required by Condition 1 in Theorem \ref{IbragimovKhasminskiiForErgDiff}.
\begin{prop}\label{Inequality1}
	Let the initial distribution be such that for any $R>0$ and for $u,v \in \mathcal{U}_{T,s}$ as defined in (\ref{DefnUTs}) with $|u|<R, |v|<R$
	\begin{align}\label{ConditionsOnNuProp3.4}
		\int_{0}^{1}\left| \nu\left(s+\frac{u}{\sqrt{T}},x\right)^{\frac{1}{2}} - \nu\left(s+\frac{v}{\sqrt{T}},x\right)^{\frac{1}{2}} \right|^{2} \lambda(dx)\leq c\left|u-v\right|^{2}
	\end{align}
	for some constant $c>0$ with dominating measure $\lambda(\cdot)$ as specified below \eqref{WFLikelihoodS}. Then for any $\mathcal{K}\subset\mathcal{S}$ compact, we can find a constant $C$ such that for any $R>0$, and for any $u,v\in\mathcal{U}_{T,s}$ as defined in (\ref{DefnUTs}) satisfying $|u|<R$, $|v|<R$, the following holds
	\begin{align*}
		\sup\limits_{s\in\mathcal{K}}\mathbb{E}^{(s)}_{\nu}\left[\Big|Z_{T,s}(u)^{\frac{1}{2}}-Z_{T,s}(v)^{\frac{1}{2}}\Big|^{2}\right] \leq C(1+R^{2}){|u-v|^{2}}.
	\end{align*}
	\noindent In particular the result holds for $\nu = f_{s}$.
\end{prop}
\begin{proof}
	In what follows we denote by $C_{i}$, for $i\in\mathbb{N}$, constants which do not depend on $u$, $v$, $s$, or $T$. Observe that for any $s', s^* \in\mathcal{S}$ it holds that
	\begin{align*}
		\mathbb{E}^{(s')}_{\nu}\left[\int_{0}^{T}\left|\frac{\mu_{\textnormal{WF}}(s',X_{t})-\mu_{\textnormal{WF}}(s^*,X_{t})}{\sigma(X_{t})}\right|^{4}dt\right] &= \mathbb{E}^{(s')}_{\nu}\left[\int_{0}^{T}\left|\frac{(s'-s^*)}{2}\sqrt{X_{t}(1-X_{t})}\right|^{4}dt\right] \nonumber \\
		&\leq \left(\frac{s'-s^*}{4}\right)^{4}T < \infty,
	\end{align*}
	and so we can use Lemma 1.13 and Remark 1.14 from \cite{Kuto} (as done in Lemma 2.10 there) to split the expectation in (\ref{IbragimovLikeRatioBoundsForErgDiff}) into three 
	\begin{align}\label{Lem1.13Applied}
		\mathbb{E}^{(s)}_{\nu}\left[\Big|Z_{T,s}^{\frac{1}{2}}(u)-Z_{T,s}^{\frac{1}{2}}(v)\Big|^{2}\right] &\leq C_{1}\int_{0}^{1}\left|\nu(s_{u},x)^{\frac{1}{2}} - \nu(s_{v},x)^{\frac{1}{2}}\right|^{2}\lambda(dx) \nonumber \\
		&\quad {} + C_{2} \int_{0}^{T}\mathbb{E}^{(s_{v})}_{\nu}\left[\left(\frac{\mu_{\textnormal{WF}}(s_{u},X_t)-\mu_{\textnormal{WF}}(s_{v},X_t)}{\sigma(X_t)}\right)^{2}\right]dt \nonumber \\
		&\quad {} + C_{3} T \int_{0}^{T} \mathbb{E}^{(s_{v})}_{\nu}\left[\left(\frac{\mu_{\textnormal{WF}}(s_{u},X_t)-\mu_{\textnormal{WF}}(s_{v},X_t)}{\sigma(X_t)}\right)^{4}\right]dt,
	\end{align}
	where we denote $s_{u} = s + {u}/{\sqrt{T}}$ and $s_{v}=s+{v}/{\sqrt{T}}$. The first term on the RHS of (\ref{Lem1.13Applied}) can be dealt with using (\ref{ConditionsOnNuProp3.4}), whilst for the second term observe that 
	\begin{align*}
		\int_{0}^{T} \mathbb{E}^{(s_{v})}_{\nu}\left[\left(\frac{\mu_{\textnormal{WF}}(s_{u},X_t)-\mu_{\textnormal{WF}}(s_{v},X_t)}{\sigma(X_t)}\right)^{2}\right]dt &= \frac{|u-v|^{2}}{4T}\int_{0}^{T}\mathbb{E}^{(s_{v})}_{\nu}\left[X_t (1-X_t)\right]dt \nonumber\\
		&\leq \frac{1}{16}|u-v|^{2}.
	\end{align*}
	Therefore
	\begin{align*}
		C_{2} \int_{0}^{T}\mathbb{E}^{(s_{v})}_{\nu}\left[\left(\frac{\mu_{\textnormal{WF}}(s_{u},X_t)-\mu_{\textnormal{WF}}(s_{v},X_t)}{\sigma(X_t)}\right)^{2}\right]dt &\leq C_{4}|u-v|^{2}.
	\end{align*}
	A similar calculation can be performed for the third term in (\ref{Lem1.13Applied}) to get
	\begin{align*}
		C_{3} T \int_{0}^{T}\mathbb{E}^{(s_{v})}_{\nu}\left[\left(\frac{\mu_{\textnormal{WF}}(s_{u},X_t)-\mu_{\textnormal{WF}}(s_{v},X_t)}{\sigma(X_t)}\right)^{4}\right]dt \leq C_{5}|u-v|^{4},
	\end{align*}
	and thus the result holds in view of the fact that $|u|, |v|<R$. \\ 
	\noindent It remains to show that \eqref{ConditionsOnNuProp3.4} holds for $\nu=f_{s}$. To this end, observe that
	\begin{multline}\label{Lemma2.10fterm}
		\int_{0}^{1}\left|f_{s_{u}}(x)^{\frac{1}{2}}-f_{s_{v}}(x)^{\frac{1}{2}}\right|^{2}dx  = \int_{0}^{1}x^{\theta_{1}-1}\left(1-x\right)^{\theta_{2}-1}e^{sx}\Big|\frac{1}{\sqrt{G_{s_{u}}}}e^{\frac{ux}{2\sqrt{T}}}-\frac{1}{\sqrt{G_{s_{v}}}}e^{\frac{vx}{2\sqrt{T}}}\Big|^{2}dx.
	\end{multline}
	Now we have that
	\begin{align*}
		C_{6}\min\{e^{s},1\} \leq G_{s_{u}} &:= \int_{0}^{1}x^{\theta_{1}-1}(1-x)^{\theta_{2}-1}e^{\left(s+\frac{u}{\sqrt{T}}\right)x}dx \leq C_{7}\max\{e^{s},1\},
	\end{align*}
	where $C_6=B(\theta_1,\theta_2)e^{-\textnormal{diam}(\mathcal{S})}$, $C_7 = B(\theta_1,\theta_2)e^{\textnormal{diam}(\mathcal{S})}$ are non-zero, positive, and independent of $s$ and $T$, since we constrain $u,v\in\mathcal{U}_{T,s}$. This allows us to deduce that $G \mapsto 1/\sqrt{G}$ is Lipschitz on $[C_{6}\inf_{s\in\mathcal{K}}\min\{e^{s},1\},C_{7}\sup_{s\in\mathcal{K}}\max\{e^{s},1\}]$ with some constant $C_{8}>0$, i.e.
	\begin{align*}
		\Bigg|\frac{1}{\sqrt{G_{s_{u}}}}-\frac{1}{\sqrt{G_{s_{v}}}}\Bigg| &\leq C_{8}\Big|G_{s_{u}} - G_{s_{v}}\Big| \\
		&= C_{8}\int_{0}^{1}x^{\theta_{1}-1}\left(1-x\right)^{\theta_{2}-1}e^{sx}\left|e^{\frac{ux}{2\sqrt{T}}}-e^{\frac{vx}{2\sqrt{T}}}\right|dx \\
		&\leq C_{8}C_{9}\int_{0}^{1}x^{\theta_{1}-1}\left(1-x\right)^{\theta_{2}-1}e^{sx}\left|\frac{ux}{2\sqrt{T}}-\frac{vx}{2\sqrt{T}}\right|dx \\
		&= \frac{C_{8}C_{9}}{2\sqrt{T}}\left|u-v\right|\int_{0}^{1}x^{\theta_{1}}\left(1-x\right)^{\theta_{2}-1}e^{sx}dx \\
		&\leq \frac{C_{10}}{\sqrt{T}}\max\{ e^{s},1 \}\left|u-v\right|,
	\end{align*}
	where in the second inequality we have made use of the fact that $e^{z}$ is Lipschitz in $z$ on $[-\textnormal{diam}(\mathcal{S}),\textnormal{diam}(\mathcal{S})]$ with some constant $C_{9}>0$. Thus we deduce that
	\allowdisplaybreaks{\begin{align}\label{Lemma2.10ftermnearlydone}
			\Big|\frac{1}{\sqrt{G_{s_{u}}}}e^{\frac{ux}{2\sqrt{T}}}-\frac{1}{\sqrt{G_{s_{v}}}}e^{\frac{vx}{2\sqrt{T}}}\Big|^{2} &= \Bigg|\frac{1}{\sqrt{G_{s_{u}}}}\left(e^{\frac{ux}{2\sqrt{T}}}-e^{\frac{vx}{2\sqrt{T}}}\right)\Bigg|^{2} +\Bigg|e^{\frac{vx}{2\sqrt{T}}}\left(\frac{1}{\sqrt{G_{s_{u}}}}-\frac{1}{\sqrt{G_{s_{v}}}}\right)\Bigg|^{2} \nonumber \\
			&\quad {} + 2\frac{e^{\frac{vx}{2\sqrt{T}}}}{\sqrt{G_{s_{u}}}}\Big|e^{\frac{ux}{2\sqrt{T}}}-e^{\frac{vx}{2\sqrt{T}}}\Big|\Big|\frac{1}{\sqrt{G_{s_{u}}}}-\frac{1}{\sqrt{G_{s_{v}}}}\Big|\nonumber \\
			&\leq \frac{C_{9}^{2}x^{2}}{4T}\frac{1}{C_{6}\min\{e^{s},1\}}\big|u-v\big|^{2}\nonumber \\
			&\quad{}+ e^{\textnormal{diam}(\mathcal{S})x}\frac{C_{10}^{2}}{T}\max\{e^{2s},1\}\big|u-v\big|^{2} \nonumber \\
			&\quad {} + \frac{e^{\textnormal{diam}(\mathcal{S})x}C_{9}C_{10}x}{T\sqrt{C_{6}}}\frac{\max\{e^{s},1\}}{\min\{e^{s/2},1\}}\big|u-v\big|^{2}.
	\end{align}}
	Putting (\ref{Lemma2.10ftermnearlydone}) into (\ref{Lemma2.10fterm}) gives us the result
	\begin{align*}
		\int_{0}^{1}x^{\theta_{1}-1}\left(1-x\right)^{\theta_{2}-1}e^{sx}\Big|\frac{1}{\sqrt{G_{s_{u}}}}e^{\frac{ux}{\sqrt{T}}}-\frac{1}{\sqrt{G_{s_{v}}}}e^{\frac{vx}{\sqrt{T}}}\Big|^{2}dx &\leq \frac{C_s}{T}|u-v|^2, 
	\end{align*}
	as
	\begin{align*}
		C_{s} := C_{11}e^{|s|} + C_{12}\max\{e^{3s},1\} + C_{13}\frac{\max\{e^{2s},1\}}{\min\{e^{s/2},1\}},
	\end{align*}
	is continuous in $s$ over any compact set $\mathcal{K}\subset\mathcal{S}$.
	
\end{proof}
\begin{prop}\label{Lemma2.11AltResult}
	Let the initial distribution be such that for any $M\geq 2$ and for $u\in\mathcal{U}_{T,s}$,
	\begin{align}\label{ConditionsOnNuProp3.5}
		\mathbb{P}_{\nu}^{(s)}\left[\left|\log\left(\frac{\nu(s+\frac{u}{\sqrt{T}},X_0)}{\nu(s,X_0)}\right)\right| > \frac{1}{48}\mathbb{E}^{(s)}\left[\xi(1-\xi)\right]\left|u\right|^{2}\right] \leq \frac{C}{|u|^{M}}
	\end{align}
	for some constant $C>0$. Then for $\mathcal{K}\subset\mathcal{S}$ compact, there exists a function $g_{T}(\cdot)\in\mathscr{G}$ such that for any $u\in \mathcal{U}_{T,s}$ as defined in (\ref{DefnUTs}) we have that
	\begin{align}\label{Lemma2.11Bound}
		\sup_{s\in\mathcal{K}}\mathbb{E}^{(s)}_{\nu}\left[Z_{T,s}(u)^{\frac{1}{2}}\right] \leq e^{-g_{T}(|u|)}.
	\end{align}
	\noindent The result holds in particular when $\nu = f_{s}$.
\end{prop}
\begin{proof}
	Assume for now that for any $M\geq 2$ we have that
	\begin{align}\label{probbound}
		\mathbb{P}^{(s)}_{\nu}\left[Z_{T,s}(u) > \exp\left\{-\frac{1}{16}\mathbb{E}^{(s)}\left[\xi(1-\xi)\right]|u|^{2}\right\} \right] \leq \frac{C_{s,M}}{|u|^{M}}
	\end{align}
	for some constant $C_{s,M}>0$ depending on $s$ and $M$. We show that if (\ref{probbound}) holds, then (\ref{Lemma2.11Bound}) follows. Indeed
	\allowdisplaybreaks{\begin{align*}
			\mathbb{E}^{(s)}_{\nu}\left[Z_{T,s}(u)^{\frac{1}{2}}\right] &= \mathbb{E}^{(s)}_{\nu}\left[Z_{T,s}(u)^{\frac{1}{2}}1_{\left\{Z_{T,s}(u) \leq \exp\left\{-\frac{1}{16}\mathbb{E}^{(s)}\left[\xi(1-\xi)\right]|u|^{2}\right\}\right\}}\right] \\
			&\quad {} + \mathbb{E}^{(s)}_{\nu}\left[Z_{T,s}(u)^{\frac{1}{2}}1_{\left\{Z_{T,s}(u) > \exp\left\{-\frac{1}{16}\mathbb{E}^{(s)}\left[\xi(1-\xi)\right]|u|^{2}\right\}\right\}}\right] \\
			&\leq \exp\left\{-\frac{1}{32}\mathbb{E}^{(s)}\left[\xi(1-\xi)\right]|u|^{2}\right\} \\
			&\quad{}+ \mathbb{E}^{(s)}_{\nu}\left[Z_{T,s}(u)\right]^{\frac{1}{2}}\mathbb{P}^{(s)}_{\nu}\left[Z_{T,s}(u) > \exp\left\{-\frac{1}{16}\mathbb{E}^{(s)}\left[\xi(1-\xi)\right]|u|^{2}\right\}\right]^{\frac{1}{2}} \\
			&\leq \exp\left\{-\frac{1}{32}\mathbb{E}^{(s)}\left[\xi(1-\xi)\right]|u|^{2}\right\} + \frac{C_{s,M}}{|u|^{\frac{M}{2}}} 
	\end{align*}}where in the first inequality we have made use of Cauchy-Schwarz, and for the second inequality we have used (\ref{probbound}). Therefore,
	\begin{align*}
		\sup_{s\in\mathcal{K}}\mathbb{E}^{(s)}_{\nu}\left[Z_{T,s}(u)^{\frac{1}{2}}\right] &\leq \sup_{s\in\mathcal{K}}\left\{  \exp\left\{-\frac{1}{32}\mathbb{E}^{(s)}\left[\xi(1-\xi)\right]|u|^{2}\right\} + \frac{C_{s,M}}{|u|^{\frac{M}{2}}} \right\} \\
		&= \exp\left\{-\frac{1}{32}\inf_{s\in\mathcal{K}}\mathbb{E}^{(s)}\left[\xi(1-\xi)\right]|u|^{2}\right\} + \frac{\sup_{s\in\mathcal{K}}C_{s,M}}{|u|^{\frac{M}{2}}}  \\
		&=: \exp\left\{-g_{T}(|u|)\right\}.
	\end{align*}
	It remains to ensure that $g_{T}(\cdot)\in\mathscr{G}$, that $\inf_{s\in\mathcal{K}}\mathbb{E}^{(s)}[\xi(1-\xi)] \geq \kappa > 0$ for some constant $\kappa$, and that for any $M\geq 2$ it holds that $\sup_{s\in\mathcal{K}}C_{s,M}<\infty$. Observe that
	\begin{align*}
		\min\left\{\inf_{s\in\mathcal{K}}e^{s},1\right\}B(\theta_{1},\theta_{2}) \leq G_{s} \leq \max\left\{\sup_{s\in\mathcal{K}}e^{s},1\right\}B(\theta_{1},\theta_{2}).
	\end{align*}
	Thus 
	\begin{align*}
		\inf_{s\in\mathcal{K}}\mathbb{E}^{(s)}\left[\xi(1-\xi)\right] &= \inf_{s\in\mathcal{K}}\left\{\int_{0}^{1}\frac{1}{G_{s}}e^{s\xi}\xi^{\theta_{1}}(1-\xi)^{\theta_{2}}d\xi\right\}\nonumber\\
		&\geq \frac{\inf_{s\in\mathcal{K}}\left\{\int_{0}^{1}e^{s\xi}\xi^{\theta_{1}}(1-\xi)^{\theta_{2}}d\xi\right\}}{\max\left\{\sup_{s\in\mathcal{K}}e^{s},1\right\}B(\theta_{1},\theta_{2})} \nonumber\\
		&\geq \frac{\min\left\{\inf_{s\in\mathcal{K}}e^{s},1\right\}B(\theta_{1}+1,\theta_{2}+1)}{\max\left\{\sup_{s\in\mathcal{K}}e^{s},1\right\}B(\theta_{1},\theta_{2})} =: \kappa
	\end{align*}
	and $\kappa>0$ because $\mathcal{K}$ is bounded, and thus both $\sup_{s\in\mathcal{K}}e^{s}$ and $\inf_{s\in\mathcal{K}}e^{s}$ are finite and non-zero. We show that $\sup_{s\in\mathcal{K}}C_{s,M}$ is finite $\forall M\geq 2$ in what follows. We now check that $g_{T}(|u|)$ as defined above is in the class of functions $\mathscr{G}$. To this end, observe that
	\begin{align*}
		g_{T}(|u|) = -\log\left(\exp\left\{-\frac{1}{32}\inf_{s\in\mathcal{K}}\mathbb{E}^{(s)}\left[\xi(1-\xi)\right]|u|^{2}\right\} + \frac{{\sup_{s\in\mathcal{K}}C_{s,M}}}{|u|^{\frac{M}{2}}}\right).
	\end{align*}
	Indeed, for a fixed $T>0$, $g_{T}(|u|)\rightarrow\infty$ as $|u|\rightarrow\infty$, because $\inf_{s\in\mathcal{K}}\mathbb{E}^{(s)}[\xi(1-\xi)]>0$, and furthermore given any fixed $N$, we can choose $M$ large enough (note the way we phrased (\ref{probbound}) allows us to choose our $M$ arbitrarily large, say $M>2N$) such that
	\begin{align*}
		\lim\limits_{\substack{T\rightarrow\infty\\y\rightarrow\infty}}y^{N}e^{-g_{T}(y)} = \lim\limits_{\substack{T\rightarrow\infty\\y\rightarrow\infty}}y^{N}\left(\exp\left\{-\frac{1}{32}\inf_{s\in\mathcal{K}}\mathbb{E}^{(s)}\left[\xi(1-\xi)\right]|y|^{2}\right\} + \frac{\sup_{s\in\mathcal{K}}C_{s,M}}{|y|^{\frac{M}{2}}}\right) = 0,
	\end{align*}
	where the order in which limits are taken is immaterial since our choice of $g_{T}(|u|)$ is independent of $T$. Thus we have proved that if (\ref{probbound}) holds, then 
	\begin{align*}
		\sup_{s\in\mathcal{K}}\mathbb{E}^{(s)}_{\nu}\left[Z_{T,s}(u)^{\frac{1}{2}}\right] \leq e^{-g_{T}(|u|)}, \hspace*{10mm}g_{T}(\cdot)\in\mathscr{G}.
	\end{align*}
	\noindent To show that (\ref{probbound}) holds, we make use of Markov's inequality as well as Theorem 3.2 in \cite{Locherbach}. Indeed, observe that 
	\begin{align*}
		\mathbb{P}^{(s)}_{\nu}&\left[Z_{T,s}(u) \geq \exp\left\{-\frac{1}{16}\mathbb{E}^{(s)}\left[\xi(1-\xi)\right]|u|^{2}\right\} \right] \\
		&= \mathbb{P}^{(s)}_{\nu}\Bigg[\frac{\nu(s+\frac{u}{\sqrt{T}},X_0)}{\nu\left(s,X_0\right)}\exp\bigg\{\frac{u}{2\sqrt{T}}\int_{0}^{T}\sqrt{X_{t}(1-X_{t})}dW_{t} \\
		&\quad{}- \frac{|u|^{2}}{8}\left(\frac{1}{T}\int_{0}^{T}X_{t}(1-X_{t})dt - \mathbb{E}^{(s)}\left[\xi(1-\xi)\right]\right)\bigg\} > \exp\left\{\frac{1}{16}\mathbb{E}^{(s)}[\xi(1-\xi)]|u|^{2}\right\} \Bigg] \\
		&\leq \mathbb{P}^{(s)}_{\nu}\left[\left|\log\left(\frac{\nu(s+\frac{u}{\sqrt{T}},X_0)}{\nu\left(s,X_0\right)}\right) \right| > \frac{1}{48}\mathbb{E}^{(s)}[\xi(1-\xi)]|u|^{2} \right] \\
		&\quad {} + \mathbb{P}^{(s)}_{\nu}\left[\left|\frac{u}{2\sqrt{T}}\int_{0}^{T}\sqrt{X_{t}(1-X_{t})}dW_{t}\right| > \frac{1}{48}\mathbb{E}^{(s)}[\xi(1-\xi)]|u|^{2} \right] \\
		&\quad {} + \mathbb{P}^{(s)}_{\nu}\left[\frac{|u|^{2}}{8}\left|\frac{1}{T}\int_{0}^{T}{X_{t}(1-X_{t})}dt - \mathbb{E}^{(s)}\left[\xi(1-\xi)\right]\right| > \frac{1}{48}\mathbb{E}^{(s)}[\xi(1-\xi)]|u|^{2} \right] \\
		& =: A_{1} + A_{2} + A_{3}.
	\end{align*}
	\noindent The bound for $A_{1}$ follows immediately from (\ref{ConditionsOnNuProp3.5}). For the particular case when $\nu=f_{s}$, we use Markov's inequality:
	\begin{align*}
		A_{1} &= \mathbb{P}^{(s)}_{\nu}\left[\left|\log\left(\frac{G_{s}}{G_{s+\frac{u}{\sqrt{T}}}}\right) + \frac{u}{\sqrt{T}}X_{0}\right| > \frac{1}{48}\mathbb{E}^{(s)}[\xi(1-\xi)]|u|^{2}\right] \\
		&\leq \left(\frac{48}{\mathbb{E}^{(s)}[\xi(1-\xi)]|u|^{2}}\right)^{M}\mathbb{E}^{(s)}_{\nu}\left[\left|\log\left(\frac{G_{s}}{G_{s+\frac{u}{\sqrt{T}}}}\right) + \frac{u}{\sqrt{T}}X_{0}\right|^{M}\right].
	\end{align*}
	But 
	\begin{align*}
		\log\left(\frac{G_{s}}{G_{s+\frac{u}{\sqrt{T}}}}\right) = \log\left(\frac{\int_{0}^{1}x^{\theta_{1}-1}(1-x)^{\theta_{2}-1}e^{sx}dx}{\int_{0}^{1}x^{\theta_{1}-1}(1-x)^{\theta_{2}-1}e^{(s+\frac{u}{\sqrt{T}})x}dx}\right) \leq {\frac{|u|}{\sqrt{T}}},
	\end{align*}
	so we have
	\begin{align*}
		A_{1} &\leq \left(\frac{48}{\mathbb{E}^{(s)}[\xi(1-\xi)]|u|^{2}}\right)^{M}\mathbb{E}^{(s)}_{\nu}\left[\left|\frac{u}{\sqrt{T}}\right|^{M}\left|1 + X_{0}\right|^{M}\right] \\
		&= \left(\frac{48}{\mathbb{E}^{(s)}[\xi(1-\xi)]\sqrt{T}|u|}\right)^{M}\mathbb{E}^{(s)}\left[\left|1 + \xi\right|^{M}\right] \\
		&\leq \left(\frac{48d_{s}}{\mathbb{E}^{(s)}[\xi(1-\xi)]|u|^{2}}\right)^{M}\mathbb{E}^{(s)}\left[\left|1 + \xi\right|^{M}\right] =: \frac{C^{(1)}_{s,M}}{|u|^{2M}},
	\end{align*}
	where in the second inequality we made use of the fact that $u\in\mathcal{U}_{T,s}$, and thus $|u|\leq d_{s}\sqrt{T}$ where we define $d_{s}:=\sup_{w\in\mathcal{\partial\mathcal{S}}}|s-w|$ (which is strictly positive and bounded as $\mathcal{S}$ is open and bounded). To see that $\sup_{s\in\mathcal{K}}C_{s,M}^{(1)}$ is bounded, observe that 
	\begin{align*}
		\sup_{s\in\mathcal{K}}C_{s,M}^{(1)} &= \sup_{s\in\mathcal{K}}\left\{\left(\frac{48d_{s}}{\mathbb{E}^{(s)}[\xi(1-\xi)]}\right)^{M}\mathbb{E}^{(s)}\left[\left|1+\xi\right|^{M}\right]\right\} \\
		&\leq \left(96\frac{B(\theta_{1},\theta_{2})}{B(\theta_{1}+1,\theta_{2}+1) }\sup_{s\in\mathcal{K}}d_{s}\frac{\max\{e^{s},1\}}{\min\{e^{s},1\}}\right)^{M},
	\end{align*}
	which is clearly finite because $\mathcal{K}$ is bounded. \\ \newline
	For $A_{2}$ we use a similar argument, but now use the fact that we have a stochastic integral:
	\begin{align*}
		A_{2} &\leq \left(\frac{48}{\mathbb{E}^{(s)}[\xi(1-\xi)]|u|^{2}}\right)^{M}\mathbb{E}^{(s)}_{\nu}\left[\left|\frac{u}{2\sqrt{T}}\int_{0}^{T}\sqrt{X_{t}(1-X_{t})}dW_{t}\right|^{M}\right]\\
		&\leq \left(\frac{24}{\mathbb{E}^{(s)}[\xi(1-\xi)]|u|}\right)^{M}\left(\frac{M}{2}(M-1)\right)^{\frac{M}{2}}T^{-1}\mathbb{E}^{(s)}_{\nu}\left[\int_{0}^{T}\left|X_{t}(1-X_{t})\right|^{\frac{M}{2}}dt\right] \\
		&\leq\left(\frac{12}{\mathbb{E}^{(s)}[\xi(1-\xi)]|u|}\right)^{M}\left(\frac{M}{2}(M-1)\right)^{\frac{M}{2}}=: \frac{C^{(2)}_{s,M}}{|u|^{M}},
	\end{align*}
	\noindent where the first line uses Markov's inequality and the second inequality uses Lemma 1.1 (equation (1.3)) in \cite{Kuto}. That $\sup_{s\in\mathcal{K}}C_{s,M}^{(2)}$ is finite follows from arguments similar to those used for the respective term in $A_1$. \\ \newline
	\noindent For $A_{3}$ we make use of Theorem 3.2 in \cite{Locherbach}, which gives us that for $M\geq 2$ 
	\begin{align}\label{LocherbachProp461}
		\mathbb{P}^{(s)}_{\nu}\Bigg[\left|\frac{1}{T}\int_{0}^{T}{X_{t}(1-X_{t})}dt - \mathbb{E}^{(s)}\left[\xi(1-\xi)\right]\right| &\geq {\frac{1}{6}\mathbb{E}^{(s)}\left[\xi(1-\xi)\right]} \Bigg] \nonumber\\
		&\leq  K(s,X,M)\frac{\|x(1-x)\|^{M}_{\infty}}{\left(\frac{\mathbb{E}^{(s)}\left[\xi(1-\xi)\right]}{6}\sqrt{T}\right)^{M}}. 
	\end{align}
	For the RHS of (\ref{LocherbachProp461}), we have that
	\begin{align*}
		K(s,X,M)\frac{\|x(1-x)\|^{M}_{\infty}}{\left(\frac{\mathbb{E}^{(s)}\left[\xi(1-\xi)\right]}{6}\sqrt{T}\right)^{M}} \leq K(s,X,M)\left(\frac{6\|x(1-x)\|_{\infty}d_{s}}{{\mathbb{E}^{(s)}\left[\xi(1-\xi)\right]}|u|}\right)^{M} =: \frac{C^{(3)}_{s,M}}{|u|^{M}},
	\end{align*}
	where $K(s,X,M)$ is a function that depends on $M$ and on the moments of the hitting times of $X$. Finally we deduce that $\sup_{s\in\mathcal{K}}C_{s,M}^{(3)}$ is finite by observing that
	\begin{align*}
		\sup_{s\in\mathcal{K}}C_{s,M}^{(3)} &= \sup_{s\in\mathcal{K}}K(s,X,M)\left(\frac{6\|x(1-x)\|_{\infty}d_{s}}{{\mathbb{E}^{(s)}\left[\xi(1-\xi)\right]}}\right)^{M} \\
		&\leq \sup_{s\in\mathcal{K}}K(s,X,M)\left(\frac{3}{2}\frac{B(\theta_{1},\theta_{2})}{B(\theta_{1}+1,\theta_{2}+1) }\sup_{s\in\mathcal{K}}d_{s}\frac{\max\{e^{s},1\}}{\min\{e^{s},1\}}\right)^{M},
	\end{align*}
	which is finite since $\|x(1-x)\|_{\infty}=1/4$, $\mathcal{K}$ is compact, and $K(s,X,M)$ is bounded by a function which is continuous in $s$ (see Appendix \ref{appendix1} for the corresponding details). 
\end{proof}
\noindent Finally, we present the result which guarantees that Condition 3 in Theorem \ref{IbragimovKhasminskiiForErgDiff} holds, and thus that the Ibragimov--Has'minskii conditions hold for the Wright--Fisher diffusion.
\begin{prop}\label{thm4.6}
	The random functions $Z_{s}(u)$ and $$\psi(v) := \int_{\mathbb{R}}\ell\left(v-u\right)\frac{Z_{s}(u)}{\int_{\mathbb{R}}Z_{s}(y)dy}du$$ attain their maximum and minimum respectively at the unique point $\bar{u} = \bar{u}_{s} = I(s)^{-1}\Delta(s)$ with probability 1.
\end{prop}
\begin{proof}
	The first assertion follows immediately from Corollary \ref{MarginalConvergence}, whilst for the second we direct the interested reader to Theorem III.2.1 in \cite{IbraKhas}, which relies on two results: Anderson's Lemma (Lemma II.10.1 in \cite{IbraKhas}), and Lemma II.10.2 in \cite{IbraKhas}.
\end{proof}

\section{Numerical Simulations}\label{Simulations}
\noindent We illustrate the results proved in Section \ref{ResultsSection} by showing consistency, convergence in distribution and convergence of moments for the MLE when applied to data simulated from the Wright--Fisher diffusion. By making use of the `exact algorithm' (see \cite{JenkinsSpano} for full details), we obtain exact draws from the Wright--Fisher diffusion. The generated paths are then used to calculate the MLE, and subsequently kernel smoothed density estimates for the rescaled MLE for various terminal times $T$ are plotted against the density of the limiting distribution. Using the definition in (\ref{MLEDefn}), the MLE for the selection parameter is given by
\begin{align*}
	\hat{s}_{T} = \frac{X_{T}-X_{0} - \int_{0}^{T}\left(-\theta_2 X_{t} + \theta_{1}(1-X_{t})\right)dt}{\int_{0}^{T}X_{t}(1-X_{t})dt},
\end{align*}
which is impossible to calculate exactly in view of the infinite-dimensional path involved in the integral. Instead we approximate the MLE by using Riemann sums instead of Lebesgue integrals, which gives rise to the approximation of $\hat{s}_{T}$ given by
\begin{align}\label{MLEApprox}
	\check{s}_{T} = \frac{X_{T}-X_{0} - \sum_{i=1}^{N}\left(-\theta_2 X_{t_i} + \theta_{1}(1-X_{t_i})\right)\Delta_i}{\sum_{i=1}^{N}X_{t_i}(1-X_{t_i})\Delta_i}
\end{align}
where $\Delta_i := t_{i}-t_{i-1}$ for a time discretisation grid $\{t_i\}_{i=0}^{N}$ where $t_0 = 0$ and $t_{N} = T$, and $N\in\mathbb{N}\setminus\{0\}$. In particular, $\{X_{t_i}\}_{i=0}^{N}$ denotes the values of the Wright--Fisher path at the times $\{t_i\}_{i=0}^{N}$, which corresponds to the output generated by the exact algorithm. \\ \newline
\noindent To simulate the Wright--Fisher paths, we set the selection parameter $s=4$, the mutation parameters $\theta_1, \theta_2 =2$ (to ensure that we are within the regime of Theorem \ref{WFEstimatorResult}), $\Delta_i = 0.001$, $X_0=0.25$ and varied the terminal time $T\in\{1,2,10,50\}$. For each of the 10,000 simulated paths, we computed (\ref{MLEApprox}), and subsequently for each $T$ we obtained kernel smoothed estimates of the density of $\sqrt{T}(\check{s}_{T}-s)$ which are plotted against the limiting $N(0,\frac{1}{4}\mathbb{E}^{(s)}[\xi(1-\xi)]^{-1})$ density in Figure \ref{Figure}.
\begin{figure}[H]
	\centering
	\includegraphics[scale=0.8]{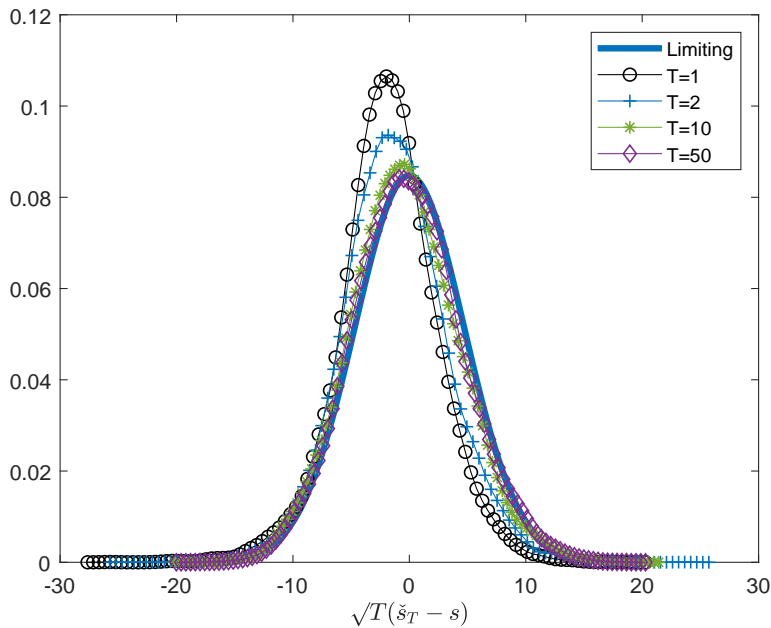}
	\caption{Plots of the kernel smoothed density estimates for $\sqrt{T}(\check{s}_{T}-s)$ for $T=1,2,10,50$, and of the limiting $N(0,I(s)^{-1})$ density.}
	\label{Figure}
\end{figure}

\section{Discussion}\label{discussion}
\noindent In this article we have provided criteria which determine whether a general diffusion defined over a bounded interval $[l,r]$ with either boundary point being either entrance or regular is $\boldsymbol{\vartheta}$-uniformly ergodic (Theorem \ref{UnifErgThm}), together with a set of additional conditions which allow for an extension of the result to a specific class of unbounded functions for diffusions possessing solely entrance boundaries (Theorem \ref{UnifErgForUnbounded}). Using the criteria in Theorem \ref{UnifErgThm}, we have shown in Corollary \ref{UnifErgThm4WF} that the Wright--Fisher diffusion is $\boldsymbol{\vartheta}$-uniformly ergodic for $\boldsymbol{\vartheta}=(s,\theta_1,\theta_2)\in\boldsymbol{\Theta}\subset\mathbb{R}\times(0,\infty)^{2}$, extending the well-known pointwise in $\boldsymbol{\vartheta}=(s,\theta_1,\theta_2)$ ergodicity of the Wright--Fisher diffusion over any compact set $\mathcal{K}\subset\mathbb{R}\times(0,\infty)^{2}$ for bounded functions. We have also proved that the family of measures $\{ \mathbb{P}^{(\boldsymbol{\vartheta})}_{\nu} : \boldsymbol{\vartheta}\in\boldsymbol{\Theta} \}$ induced by the solution to the SDE (\ref{WFDiff}) are uniformly LAN when $\boldsymbol{\Theta}\subset\mathbb{R}\times[1,\infty)^{2}$ in Theorem \ref{WFDifLAN} (by making use of Corollary \ref{UnifErgForUnboundedWF} which uses the conditions derived in Theorem \ref{UnifErgForUnbounded}), where the extra restriction on the mutation rates ensures that the likelihood ratio function is defined. \\ \newline
\noindent In Section 3 we then considered inference for the selection parameter $s$ when the diffusion is observed continuously through time and the mutation rates are known. Under these assumptions, we proved that the ML and Bayesian estimators for $s\in\mathcal{S}$ ($\mathcal{S}$ an open bounded subset of $\mathbb{R}$) in the non-neutral Wright--Fisher diffusion started from a broad class of initial distributions which includes its stationary distribution, are uniformly over compact sets consistent and display uniform in $s\in\mathcal{K}$ asymptotic normality and convergence of moments, for any compact $\mathcal{K}\subset\mathcal{S}$. Furthermore, for the right choice of loss function we also have asymptotic efficiency of the two estimators. The uniformity in these results is particularly useful as it guarantees a lower bound on the rate at which the inferential parameters are being learned. Such properties have been shown to hold for a wide class of SDEs in \cite{Kuto} by making use of the general theorems of Ibragimov and Has'minskii (Theorems I.5.1, I.5.2, I.10.1 and I.10.2 in \cite{IbraKhas}), however they do not hold for the Wright--Fisher diffusion as they require the diffusion coefficient to be non-zero everywhere and to have an inverse that has a polynomial majorant. Both conditions fail for (\ref{WFDiff}), forcing us to find an alternative way of proving that the Ibragimov--Has'minskii conditions still hold. We emphasise here that the aim of this study is to investigate the properties of the estimators in the ``ideal" continuous observation scenario when the whole path is known to the observer. \\ \newline 
\noindent Assuming that the mutation rates are known is a limitation to this study, however we emphasise that in the regime considered here these can be inferred directly from the path once the diffusion gets arbitrarily close to either boundary (see Remark \ref{InferringMutation} for the corresponding details). Nonetheless, extending this work to include mutation parameters greater than 1 would be of great interest. This proves to be rather challenging as now the likelihood ratio function involves expressions of the form $(1-x){x}^{-1}$ and $x(1-x)^{-1}$ (as witnessed in Theorem \ref{WFDifLAN}) which require much more delicate arguments in order to establish the same conclusions as in Theorem \ref{WFEstimatorResult}. The main issue here is in showing that Condition 1 in the Ibragimov--Has'minskii conditions holds, for the other two conditions follow from Theorem \ref{WFDifLAN} and Proposition \ref{thm4.6}. In particular, the fact that the functions $(1-x){x}^{-1}$ and $x(1-x)^{-1}$ are unbounded in $x$ and have only finitely many moments with respect to the stationary distribution means that the strategies used in the proofs of Propositions \ref{Inequality1} and \ref{Lemma2.11AltResult} cannot be used. \\ \newline
\noindent Recent advances in genome sequencing technology have led to an increase in the availability and analysis of genetic time series data. Inference for selection has traditionally been conducted using techniques for and data coming from a single point in time. However, having a time series of data points allows one to track the changes in allele frequencies over time, to better understand and infer the presence and effect of selection. Several inferential techniques have already been developed for such a setting (see for instance \cite{Bollback, Malaspinas, Schraiber, Beaumont1, Beaumont2}, as well as \cite{Dehasque} for a review on the subject), and although the techniques provide ostensibly reasonable estimation, there are not always theoretical guarantees on the statistical properties of the estimators being used. The results presented in this paper offer a baseline in this regard, and prove that in the absence of observational error one is guaranteed that the ML and Bayesian estimators are uniform over compact sets consistent, asymptotically normal, and display moment convergence, besides being asymptotically efficient for the right choice of loss function.  

\section{Acknowledgements}
\noindent This work was supported by the EPSRC and the MASDOC DTC (under grant EP/HO23364/1), by The Alan Turing Institute under the EPSRC grant EP/N510129/1, and by the EPSRC under grant EP/R044732/1. 

\appendix

\section{Proof of Theorem \ref{UnifErgThm}}\label{appendix1}
\begin{proof} 
	We show $\boldsymbol{\vartheta}$-uniform ergodicity for scalar diffusions on the bounded interval $[l,r]$ having entrance or regular boundary points by making use of Theorem 3.2 in \cite{Locherbach}, which allows us to bound the LHS of (\ref{UnifErgDef}) in terms of the moments of the hitting times of the process. That result requires the diffusion coefficient to be positive everywhere, and the drift and diffusion coefficients to be locally Lipschitz and to satisfy a linear growth condition. These conditions however, are only used to guarantee the existence of a unique strong non-exploding solution to the SDE in Theorem 3.2, which we are guaranteeing explicitly in the statement of the theorem. None of these requirements on the drift and diffusion coefficients are used in the proof of Theorem 3.2 in \cite{Locherbach} when $p\in\{2,3,\dots\}$, which allows us to employ this theorem for such $p$. All that remains to prove then is that these moments can be bounded in $\boldsymbol{\vartheta}$ over compact sets in the parameter space, for then (\ref{UnifErgDef}) holds. To this end, we introduce some notation from \cite{Locherbach}, namely let $a,b\in(l,r)$ be arbitrary fixed points such that $a<b$. Define $S_{0}=0$, $R_{0}=0$, and
	\begin{align*}
		S_{1} &:= \inf\left\{ t\geq 0 : Y_{t} = b \right\}\\
		R_{1} &:= \inf\left\{ t \geq S_{1} : Y_{t} = a \right\}\\
		S_{n+1} &:= \inf\left\{ t \geq R_{n} : Y_{t} = b \right\}\\
		R_{n+1} &:= \inf\left\{ t \geq S_{n+1} : Y_{t} = a \right\}
	\end{align*}
	for $n\in\mathbb{N}$. By the strong Markov property, $(R_{k}-R_{k-1})_{k\in\mathbb{N}\setminus\{0\}}$ is an i.i.d.\ sequence with law under $\mathbb{P}^{(\boldsymbol{\vartheta})}_{\nu}$ equal to the law of $R_{1}$ under $\mathbb{P}^{(\boldsymbol{\vartheta})}_{a}$, where $\mathbb{P}^{(\boldsymbol{\vartheta})}_{\nu}$ and $\mathbb{E}^{(\boldsymbol{\vartheta})}_{\nu}$ are as defined in Section 2, and $\mathbb{P}^{(\boldsymbol{\vartheta})}_{a}$ denotes the law of the process started from $a$. Related to the process $(R_{n})_{n\in\mathbb{N}}$ we have the process $(N_{t})_{t\geq 0}$ which we define as
	\begin{align*}
		N_{t} := \sup\left\{ n : R_{n} \leq t \right\}
	\end{align*}
	and for which we observe that $\{N_{t} \geq n\} = \{ R_{n} \leq t \}$. We also denote by
	\begin{align*}
		T_{b} := \inf\{t\geq 0 : Y_{t} = b \}
	\end{align*}
	the hitting time of $b$. Furthermore, let $\ell_{\boldsymbol{\vartheta}} := \mathbb{E}^{(\boldsymbol{\vartheta})}[N_{1}] = {\mathbb{E}^{(\boldsymbol{\vartheta})}_{a}[R_{1}]}^{-1}$ (see Lemma 2.7 in \cite{Locherbach}), and $\bar\eta_{1} := -(R_{2}-R_{1}-\ell_{\boldsymbol{\vartheta}}^{-1})$. Then Theorem 3.2 in \cite{Locherbach} gives us that for $p\in\{2,3,\dots\}$
	\begin{align*}
		\mathbb{P}^{(\boldsymbol{\vartheta})}_{\nu}\left[\left|\frac{1}{T}\int_{0}^{T}h(Y_{t})dt - \mathbb{E}^{(\boldsymbol{\vartheta})}\left[h(\xi)\right]\right| > \varepsilon \right] \leq 
		K(\boldsymbol{\vartheta},Y,p)\varepsilon^{-p}\|h\|_{\infty}^{p}T^{-\frac{p}{2}},
	\end{align*}
	where
	\begin{align*}
		K(\boldsymbol{\vartheta},Y,p) &:= {}6^{\frac{p}{2}}\mathbb{E}^{(\boldsymbol{\vartheta})}_{\nu}\left[R_{1}^{\frac{p}{2}}\right] + 12^{p}C_{p}\ell_{\boldsymbol{\vartheta}}^{\frac{p}{2}}\mathbb{E}^{(\boldsymbol{\vartheta})}_{\nu}\left[\left|R_{2}-R_{1}\right|^{p}\right]+2(6^p)\ell_{\boldsymbol{\vartheta}}\mathbb{E}^{(\boldsymbol{\vartheta})}_{a}\left[R_{1}^{p}\right] \nonumber \\
		&\quad{}+ 2^{\frac{p}{2}}\mathbb{E}^{(\boldsymbol{\vartheta})}_{\nu}\left[\left|R_{1}-{\ell_{\boldsymbol{\vartheta}}}^{-1}\right|^{\frac{p}{2}}\right] + 2^{\frac{3p}{2}}C_{p}\ell_{\boldsymbol{\vartheta}}^{\frac{p}{2}}\mathbb{E}^{(\boldsymbol{\vartheta})}_{\nu}\left[\left|\bar{\eta}_{1}\right|^{p}\right], 
	\end{align*}
	and $C_{p}$ is a constant depending only on $p$. We point out here that Theorem 3.2 in \cite{Locherbach} holds $\forall p\in(1,\infty)$ under additional assumptions, but for our case we need only $p\in\{2,3,\dots\}$. Thus we are left with showing these moments can be bounded from above in $\boldsymbol{\vartheta}$ over compact sets, for then (\ref{UnifErgDef}) follows. Now the only terms above that depend on $\boldsymbol{\vartheta}$ are
	\begin{align}
		\mathbb{E}^{(\boldsymbol{\vartheta})}_{\nu}\left[R_{1}^{\frac{p}{2}}\right], & & \ell_{\boldsymbol{\vartheta}}^{\frac{p}{2}}\mathbb{E}^{(\boldsymbol{\vartheta})}_{\nu}\left[\left|R_{2}-R_{1}\right|^{{p}}\right], & & \ell_{\boldsymbol{\vartheta}}\mathbb{E}^{(\boldsymbol{\vartheta})}_{a}\left[R_{1}^{{p}}\right], & & \mathbb{E}^{(\boldsymbol{\vartheta})}_{\nu}\left[\left|R_{1}-\ell_{\boldsymbol{\vartheta}}^{-1}\right|^{\frac{p}{2}}\right], & & \ell_{\boldsymbol{\vartheta}}^{\frac{p}{2}}\mathbb{E}^{(\boldsymbol{\vartheta})}_{\nu}\left[\left|\bar{\eta}_{1}\right|^{{p}}\right]
	\end{align}
	and in light of the following inequalities 
	\allowdisplaybreaks{
		\begin{align*}
			\mathbb{E}^{(\boldsymbol{\vartheta})}_{\nu}\left[\left|\bar{\eta}_{1}\right|^{{p}}\right] &\leq 2^{p-1}\left(\mathbb{E}^{(\boldsymbol{\vartheta})}_{\nu}\left[\left|R_{2}-R_{1}\right|^{{p}}\right] + \mathbb{E}^{(\boldsymbol{\vartheta})}_{\nu}\left[\ell_{\boldsymbol{\vartheta}}^{{-p}}\right]\right) \nonumber \\
			&= 2^{p-1}\left(\mathbb{E}^{(\boldsymbol{\vartheta})}_{a}\left[R_{1}^{{p}}\right] + \mathbb{E}^{(\boldsymbol{\vartheta})}_{a}\left[R_{1}\right]^{{p}}\right), \\
			\mathbb{E}^{(\boldsymbol{\vartheta})}_{\nu}\left[\left|R_{1}-\ell_{\boldsymbol{\vartheta}}^{-1}\right|^{\frac{p}{2}}\right] &\leq 2^{\frac{p}{2}-1}\left(\mathbb{E}^{(\boldsymbol{\vartheta})}_{\nu}\left[R_{1}^{\frac{p}{2}}\right] + \mathbb{E}^{(\boldsymbol{\vartheta})}_{\nu}\left[\ell_{\boldsymbol{\vartheta}}^{-\frac{p}{2}}\right]\right) \\
			&= 2^{\frac{p}{2}-1}\left(\mathbb{E}^{(\boldsymbol{\vartheta})}_{\nu}\left[R_{1}^{\frac{p}{2}}\right] + \mathbb{E}^{(\boldsymbol{\vartheta})}_{a}\left[R_{1}\right]^{\frac{p}{2}}\right),  \\
			\mathbb{E}^{(\boldsymbol{\vartheta})}_{\nu}\left[\left|R_{2}-R_{1}\right|^{{p}}\right] &= \mathbb{E}^{(\boldsymbol{\vartheta})}_{a}\left[R_{1}^{p}\right] \leq 2^{p-1}\left(\mathbb{E}^{(\boldsymbol{\vartheta})}_{a}\left[T_{b}^{{p}}\right] + \mathbb{E}^{(\boldsymbol{\vartheta})}_{b}\left[T_{a}^{p}\right]\right), \\
			\mathbb{E}^{(\boldsymbol{\vartheta})}_{\nu}\left[R_{1}^{\frac{p}{2}}\right] &\leq 2^{\frac{p}{2}-1}\left(\mathbb{E}^{(\boldsymbol{\vartheta})}_{\nu}\left[T_{b}^{\frac{p}{2}}\right] + \mathbb{E}^{(\boldsymbol{\vartheta})}_{b}\left[T_{a}^{\frac{p}{2}}\right]\right), \\
			\mathbb{E}^{(\boldsymbol{\vartheta})}_{a}\left[R_{1}\right] &= \mathbb{E}^{(\boldsymbol{\vartheta})}_{a}\left[T_{b}\right] + \mathbb{E}^{(\boldsymbol{\vartheta})}_{b}\left[T_{a}\right],
	\end{align*}}
	\noindent it suffices to consider only the terms $\ell_{\boldsymbol{\vartheta}}$ and $\mathbb{E}^{(\boldsymbol{\vartheta})}_{\nu}\left[T_{b}^{p}\right]$. Thus we are left with showing that these two terms can be bounded from above in $\boldsymbol{\vartheta}$ over any compact set $\mathcal{K}\subset\boldsymbol{\Theta}$. We further point out that we can reduce our considerations in the expressions above to integer moments, for if this is not the case then
	\begin{align*}
		\mathbb{E}^{(\boldsymbol{\vartheta})}_{\nu}\left[T_{b}^{p}\right] \leq \mathbb{E}^{(\boldsymbol{\vartheta})}_{\nu}\left[T_{b}^{\lceil p \rceil}\right] + \mathbb{E}^{(\boldsymbol{\vartheta})}_{\nu}\left[T_{b}^{\lfloor p \rfloor}\right]
	\end{align*}
	where $\lceil\cdot\rceil$ and $\lfloor\cdot\rfloor$ denote the ceiling and floor functions respectively. \\ \newline
	\noindent We make use of the backward equation for the quantity $U_{q,b}(x)  := \mathbb{E}^{(\boldsymbol{\vartheta})}_{x}[T_{b}^{q}]$ for $q\in\{1,2,\dots\}$, to derive the ODE (as can be found in \cite{KarlinTaylor} p.\ 203 and 210, and \cite{WangChuancun})
	\begin{align}\label{ODE}
		\frac{\sigma^{2}(x)}{2}U_{q,b}''(x) + \mu(\boldsymbol{\vartheta},x)U_{q,b}'(x) + qU_{q-1,b}(x) = 0
	\end{align}
	with boundary conditions $U_{q,b}(b) = 0$ and 
	\begin{align*}
		\lim_{y\rightarrow l}S'(y)^{-1}\frac{\partial}{\partial y}U_{q,b}(y) = 0
	\end{align*}
	\noindent when $x<b$, or 
	\begin{align*}
		\lim_{y\rightarrow r}S'(y)^{-1}\frac{\partial}{\partial y}U_{q,b}(y) = 0
	\end{align*}
	\noindent when $x>b$, where
	\begin{align*}
		S(x) := \int^{x}e^{-\int^{y}\frac{2\mu(z)}{\sigma^{2}(z)}dz}dy.
	\end{align*}
	\noindent Solving (\ref{ODE}) for $x<b$ leads to
	\begin{align}\label{case1Exb1}
		\mathbb{E}^{(\boldsymbol{\vartheta})}_{x}[T_{b}^{q}] = \int_{x}^{b}e^{-\int^{\xi}\frac{2\mu(\boldsymbol{\vartheta},y)}{\sigma^{2}(y)}dy}\int_{l}^{\xi}\frac{2}{\sigma^{2}(\eta)}e^{\int^{\eta}\frac{2\mu(\boldsymbol{\vartheta},y)}{\sigma^{2}(y)}dy}qU_{q-1,b}(\eta)d\eta d\xi,
	\end{align}
	whilst for $x>b$ we have that
	\begin{align}\label{case1Exb2}
		\mathbb{E}^{(\boldsymbol{\vartheta})}_{x}[T_{b}^{q}] = \int_{b}^{x}e^{-\int^{\xi}\frac{2\mu(\boldsymbol{\vartheta},y)}{\sigma^{2}(y)}dy}\int_{\xi}^{r}\frac{2}{\sigma^{2}(\eta)}e^{\int^{\eta}\frac{2\mu(\boldsymbol{\vartheta},y)}{\sigma^{2}(y)}dy}qU_{q-1,b}(\eta)d\eta d\xi.
	\end{align}
	We claim that for any $x<b$ and any $q\in\{1,2,\dots\}$,
	\begin{align}\label{Integralboundx<b}
		\mathbb{E}^{(\boldsymbol{\vartheta})}_{x}[T_{b}^{q}] &\leq q!\left(\int_{l}^{b}e^{-\int^{\xi}\frac{2\mu(\boldsymbol{\vartheta},y)}{\sigma^{2}(y)}dy}\int_{l}^{\xi}\frac{2}{\sigma^{2}(\eta)}e^{\int^{\eta}\frac{2\mu(\boldsymbol{\vartheta},y)}{\sigma^{2}(y)}dy}d\eta d\xi\right)^{q} \nonumber \\
		&= q!\kappa^{l}_{\boldsymbol{\vartheta}}(l,b)^{q} < \infty.
	\end{align}
	To see this, observe that 
	\begin{align}\label{IntegralBounds}
		\mathbb{E}^{(\boldsymbol{\vartheta})}_{x}[T_{b}] &= \int_{x}^{b}e^{-\int^{\xi}\frac{2\mu(\boldsymbol{\vartheta},y)}{\sigma^{2}(y)}dy}\int_{l}^{\xi}\frac{2}{\sigma^{2}(\eta)}e^{\int^{\eta}\frac{2\mu(\boldsymbol{\vartheta},y)}{\sigma^{2}(y)}dy}d\eta d\xi \nonumber \\
		&\leq \int_{l}^{b}e^{-\int^{\xi}\frac{2\mu(\boldsymbol{\vartheta},y)}{\sigma^{2}(y)}dy}\int_{l}^{\xi}\frac{2}{\sigma^{2}(\eta)}e^{\int^{\eta}\frac{2\mu(\boldsymbol{\vartheta},y)}{\sigma^{2}(y)}dy}d\eta d\xi \nonumber\\
		&= \kappa^{l}_{\boldsymbol{\vartheta}}(l,b),
	\end{align}
	and we observe that $\kappa^{l}_{\boldsymbol{\vartheta}}(l,b)$ is finite for all $\boldsymbol{\vartheta}\in\boldsymbol{\Theta}$ in virtue of $l$ being either an entrance or regular boundary (see Table 6.2 in \cite[Chapter 15, Section 6]{KarlinTaylor} p. 234, and note that $\kappa^{l}_{\boldsymbol{\vartheta}}(l,b)$ here corresponds to $N(l)$ as defined in (6.19) there). Observe that the RHS of \eqref{IntegralBounds} is independent of $x$, so we can use the recursion in (\ref{case1Exb2}) to conclude by induction that (\ref{Integralboundx<b}) holds for $q\in\{1,2,\dots\}$ as required. Similar arguments to those presented above coupled with the requirement that the boundary point at $r$ is either entrance or regular, allows us to conclude that for $x>b$ and $q\in\{1,2,\dots\}$,
	\begin{align}\label{Integralboundsx>b}
		\mathbb{E}^{(\boldsymbol{\vartheta})}_{x}[T_{b}^{q}] &\leq q!\left(\int_{b}^{r}e^{-\int^{\xi}\frac{2\mu(\boldsymbol{\vartheta},y)}{\sigma^{2}(y)}dy}\int_{\xi}^{r}\frac{2}{\sigma^{2}(\eta)}e^{\int^{\eta}\frac{2\mu(\boldsymbol{\vartheta},y)}{\sigma^{2}(y)}dy}d\eta d\xi\right)^{q} \nonumber \\
		&= q!\kappa^{r}_{\boldsymbol{\vartheta}}(b,r)^{q} < \infty.
	\end{align}
	Both RHS of (\ref{Integralboundx<b}) and (\ref{Integralboundsx>b}) are independent of $x$, so trivially
	\begin{align}\label{IntegralBoundfs}
		\mathbb{E}^{(\boldsymbol{\vartheta})}_{\nu}\left[T_{b}^{q}\right] &\leq q!\left(\kappa^{l}_{\boldsymbol{\vartheta}}(l,b)^{q} + \kappa^{r}_{\boldsymbol{\vartheta}}(b,r)^{q}\right).
	\end{align}
	All the terms on the RHS of (\ref{Integralboundx<b}), (\ref{Integralboundsx>b}) and (\ref{IntegralBoundfs}) are finite for $\boldsymbol{\vartheta}\in\boldsymbol{\Theta}$, so we have our required bound when taking the supremum over a compact set $\mathcal{K}\subset\boldsymbol{\Theta}$ for $\mathbb{E}^{(\boldsymbol{\vartheta})}_{\nu}\left[T_{b}^{q}\right]$. It remains to show that we can bound $\ell_{\boldsymbol{\vartheta}}$ from above. Observe that by definition
	\begin{align*}
		\ell_{\boldsymbol{\vartheta}} = \mathbb{E}^{(\boldsymbol{\vartheta})}_{a}\left[R_{1}\right]^{-1} = \left(\mathbb{E}^{(\boldsymbol{\vartheta})}_{a}\left[T_{b}\right] + \mathbb{E}^{(\boldsymbol{\vartheta})}_{b}\left[T_{a}\right]\right)^{-1},
	\end{align*}
	and recall that we will take the supremum in $\boldsymbol{\vartheta}$ over a given compact set $\mathcal{K}$. Using (\ref{case1Exb1}) and (\ref{case1Exb2}) with $q=1$, coupled with \eqref{ellConditions}, we deduce that $\mathbb{E}^{(\boldsymbol{\vartheta})}_{a}\left[T_{b}\right]$ and $\mathbb{E}^{(\boldsymbol{\vartheta})}_{b}\left[T_{a}\right]$ are bounded away from 0 for any compact $\mathcal{K}\subset\boldsymbol{\Theta}$, and thus we have the required upper bound on $\ell_{\boldsymbol{\vartheta}}$.
\end{proof}

\section{Proof of Theorem \ref{UnifErgForUnbounded}}\label{appendix2}
\begin{proof}
	\noindent Recall the notation introduced in Appendix \ref{appendix1}, namely the regeneration times $\{ S_n, R_n \}_{n=0}^{\infty}$ and the number of upcrossings up to time $t$, $\{ N_t \}_{t\geq0}$. We want to prove that 
	\begin{align}\label{Appendix2Aim}
		\lim\limits_{T\rightarrow\infty}\sup\limits_{\boldsymbol{\vartheta}\in\mathcal{K}}\mathbb{P}^{(\boldsymbol{\vartheta})}_{\nu}\left[\left|\frac{1}{T}\int_{0}^{T}h(Y_t)dt - \mathbb{E}^{(\boldsymbol{\vartheta})}\left[h(\xi)\right]\right| > \varepsilon \right] = 0
	\end{align} 
	holds for any compact set $\mathcal{K}\subset\boldsymbol{\Theta}$, with $h$ as defined in the statement of the theorem. The strategy here will be to decompose the sample path of the diffusion into i.i.d.\ blocks of excursions as done in Theorem 3.5 in \cite{Locherbach}. However, we will deal with the resulting expectations in a different way, namely by applying the ODE approach used in Appendix \ref{appendix1} to bound these quantities from above in $\boldsymbol{\vartheta}$ over a compact set $\mathcal{K}$. To this end, fix $\varepsilon\in(0,\mathbb{E}^{(\boldsymbol{\vartheta})}[h(\xi)])$ and choose $\delta\in(0,1)$ such that $\varepsilon=\delta\mathbb{E}^{(\boldsymbol{\vartheta})}[h(\xi)]$, and set $\Omega_T := \{ |N_T T^{-1} - \ell_{\boldsymbol{\vartheta}}|\leq \ell_{\boldsymbol{\vartheta}}\delta/4 \}$ for $\ell_{{\boldsymbol{\vartheta}}}=\mathbb{E}_{a}^{(\boldsymbol{\vartheta})}[R_1]^{-1}$. Then as in the proof of Theorem 3.5 in \cite{Locherbach}, we get the following decomposition
	\allowdisplaybreaks{\begin{align*}
			\mathbb{P}_{\nu}^{(\boldsymbol{\vartheta})}&\left[\left|\frac{1}{T}\int_{0}^{T}h(Y_t)dt - \mathbb{E}^{(\boldsymbol{\vartheta})}\left[h(\xi)\right]\right| >\varepsilon\right] \\
			\leq{} & \mathbb{P}_{\nu}^{(\boldsymbol{\vartheta})}\left[\left|\int_{0}^{R_1}h(Y_t)dt \right|>\frac{T\varepsilon}{4}\right] \\
			&{}+ \mathbb{P}_{\nu}^{(\boldsymbol{\vartheta})}\left[\left|\int_{R_1}^{R_{N_T}+1}h(Y_t)dt - N_T \mathbb{E}^{(\boldsymbol{\vartheta})}\left[h(\xi)\right]\mathbb{E}_{a}^{(\boldsymbol{\vartheta})}\left[R_1\right]\right|>\frac{T\varepsilon}{4} ; \Omega_T\right] \\
			&{}+ \mathbb{P}_{\nu}^{(\boldsymbol{\vartheta})}\left[\left|N_T \mathbb{E}^{(\boldsymbol{\vartheta})}\left[h(\xi)\right]\mathbb{E}_{a}^{(\boldsymbol{\vartheta})}\left[R_1\right]-T\mathbb{E}^{(\boldsymbol{\vartheta})}\left[h(\xi)\right]\right|>\frac{T\varepsilon}{4} ; \Omega_T\right] \\
			&{}+ \mathbb{P}_{\nu}^{(\boldsymbol{\vartheta})}\left[\left|\int_{T}^{R_{N_T}+1}h(Y_t)dt \right|>\frac{T\varepsilon}{4} ; \Omega_T\right] + \mathbb{P}_{\nu}^{(\boldsymbol{\vartheta})}\left[\Omega_T ^{c}\right] =: A+B+E+C+D
	\end{align*}}\noindent Dealing with $E$ and $D$ can be achieved as in equations (3.10) and (3.14) in \cite{Locherbach}, to deduce that $E=0$ and
	\begin{align*}
		D \leq \frac{1}{T\varepsilon^{2}}\mathbb{E}^{(\boldsymbol{\vartheta})}\left[h(\xi)\right]^{2}\left(2\mathbb{E}_{\nu}^{(\boldsymbol{\vartheta})}\left[\left|R_1 - \ell_{\boldsymbol{\vartheta}}^{-1}\right|\right] + 2^{3}C_{1}^{2}\mathbb{E}^{(\boldsymbol{\vartheta})}_{\nu}\left[\left|\bar{\eta}_1\right|^{2}\right]\ell_{\boldsymbol{\vartheta}}\right),
	\end{align*}
	for $C_1$ the constant from the Burkholder-Davis-Gundy inequality. All the above expressions are either constant or have been shown to be bounded in $\boldsymbol{\vartheta}$ over compact sets in the parameter space in Appendix \ref{appendix1}, so it remains to deal with terms $A$, $B$ and $C$ above. \\ \newline
	\noindent Applying Markov's inequality to $A$ gives
	\begin{align*}
		A \leq \frac{4}{T\varepsilon}\mathbb{E}^{(\boldsymbol{\vartheta})}_{\nu}\left[\int_{0}^{R_1}h(Y_t)dt\right]
	\end{align*}
	and we can decompose the above integral
	\begin{align}\label{Aterm}
		\mathbb{E}^{(\boldsymbol{\vartheta})}_{\nu}\left[\int_{0}^{R_1}h(Y_t)dt\right] &= \mathbb{E}^{(\boldsymbol{\vartheta})}_{\nu}\left[\int_{0}^{S_1}h(Y_t)dt\right] + \mathbb{E}^{(\boldsymbol{\vartheta})}_{\nu}\left[\int_{S_1}^{R_1}h(Y_t)dt\right] \nonumber\\
		&\leq \mathbb{E}^{(\boldsymbol{\vartheta})}_{\nu}\left[\int_{0}^{T_b}h(Y_t)dt\right] + \sup_{y\in[a,b]}h(y)\mathbb{E}^{(\boldsymbol{\vartheta})}_{\nu}\left[R_1\right].
	\end{align}
	So it remains to prove that the first term on the RHS can be bounded in $\boldsymbol{\vartheta}$. It turns out that $B$ and $C$ can be bounded by similar quantities, so we do this first and subsequently show that the resulting quantities can be bounded in $\boldsymbol{\vartheta}$ too.\\ \newline
	\noindent Indeed, set $\xi_k := \int_{R_k}^{R_{k+1}}h(Y_t)dt$, $M_0=0$, and 
	\begin{align*}
		M_n := \sum_{k=1}^{n}\left(\xi_k - \mathbb{E}^{(\boldsymbol{\vartheta})}_{\nu}\left[\xi_k\right]\right).
	\end{align*}
	Then
	\begin{align*}
		B = \mathbb{P}^{(\boldsymbol{\vartheta})}_{\nu}\left[\left|M_{N_T}\right| > \frac{T\varepsilon}{4} ; \Omega_T\right] &\leq \mathbb{P}^{(\boldsymbol{\vartheta})}_{\nu}\left[\sup\limits_{n\leq \lfloor{T\ell_{\boldsymbol{\vartheta}}(1+\delta/4)}\rfloor}\left|M_{n}\right| > \frac{T\varepsilon}{4} \right] \\
		&\leq \left(\frac{4}{T\varepsilon}\right)^{2}\mathbb{V}^{(\boldsymbol{\vartheta})}_{\nu}\left[M_{\lfloor{T\ell_{\boldsymbol{\vartheta}}(1+\delta/4)}\rfloor}\right] 
	\end{align*}
	by the Kolmogorov inequality where $\mathbb{V}^{(\boldsymbol{\vartheta})}_{\nu}$ denotes the variance with respect to the measure $\mathbb{P}_{\nu}^{(\boldsymbol{\vartheta})}$. Now observe that
	\begin{align*}
		\mathbb{V}^{(\boldsymbol{\vartheta})}_{\nu}\left[M_{\lfloor{T\ell_{\boldsymbol{\vartheta}}(1+\delta/4)}\rfloor}\right] &= \sum_{k=1}^{\lfloor{T\ell_{\boldsymbol{\vartheta}}(1+\delta/4)}\rfloor}\mathbb{V}_{\nu}^{\boldsymbol{\vartheta}}\left[\left(\xi_{k} - \mathbb{E}_{\nu}^{(\boldsymbol{\vartheta})}\left[\xi_k\right]\right)^{2}\right]\\ 
		&= \lfloor{T\ell_{\boldsymbol{\vartheta}}(1+\delta/4)}\rfloor\mathbb{E}^{(\boldsymbol{\vartheta})}_{\nu}\left[\left(\xi_1 - \mathbb{E}^{(\boldsymbol{\vartheta})}_{\nu}\left[\xi_1\right]\right)^{2}\right] \\
		&\leq \lfloor{T\ell_{\boldsymbol{\vartheta}}(1+\delta/4)}\rfloor2\left(\mathbb{E}^{(\boldsymbol{\vartheta})}_{a}\left[\xi_0 ^{2}\right] + \mathbb{E}^{(\boldsymbol{\vartheta})}_{a}\left[\xi_0\right]^{2}\right).
	\end{align*}
	because the $\{\xi_k\}_{k=1}^{\infty}$ are i.i.d., and moreover we have that under $\mathbb{P}^{(\boldsymbol{\vartheta})}_{\nu}$ they are equal in distribution to $\xi_0$ under $\mathbb{P}^{(\boldsymbol{\vartheta})}_{a}$. So
	\begin{align}
		\label{Bbound}
		B &\leq \frac{4^{2}\lfloor{\ell_{\boldsymbol{\vartheta}}(1+\delta/4)}\rfloor}{T\varepsilon^{2}}2\left(\mathbb{E}^{(\boldsymbol{\vartheta})}_{a}\left[\xi_0 ^{2}\right] + \mathbb{E}^{(\boldsymbol{\vartheta})}_{a}\left[\xi_0\right]^{2}\right).
	\end{align}
	The second term of \eqref{Bbound} can be bounded in the same way as in (\ref{Aterm}), whilst for the first term we can use a similar decomposition to get
	\begin{align}\label{BoundingSecondMoment}
		\mathbb{E}^{(\boldsymbol{\vartheta})}_{a}\left[\xi_0 ^{2}\right] \leq 2\left(\mathbb{E}^{(\boldsymbol{\vartheta})}_{a}\left[\left(\int_{0}^{T_b}h(Y_t)dt\right)^{2}\right] + \sup_{y\in[a,b]}h(y)^{2}\mathbb{E}^{(\boldsymbol{\vartheta})}_{a}\left[R_{1}^{2}\right]\right). 
	\end{align}
	Finally, for $C$ we use the same arguments as in \cite{Locherbach} (just before equation (3.13)) to get that
	\begin{align*}
		C &\leq \sum_{k=1}^{\lfloor T\ell_{\boldsymbol{\vartheta}}(1+\delta/4)\rfloor}\mathbb{P}^{(\boldsymbol{\vartheta})}_{\nu}\left[\int_{R_k}^{R_{k}+1}h(Y_t)dt>\frac{T\varepsilon}{4}\right] \\
		&\leq \frac{\lfloor T\ell_{\boldsymbol{\vartheta}}(1+\delta/4)\rfloor}{T^2 \varepsilon^2} \mathbb{E}^{(\boldsymbol{\vartheta})}_{\nu}\left[\left(\int_{R_1}^{R_2}h(Y_t)dt\right)^2\right]\\
		&\leq \frac{\ell_{\boldsymbol{\vartheta}}(1+\delta/4)}{T \varepsilon^2} \mathbb{E}^{(\boldsymbol{\vartheta})}_{a}\left[\left(\int_{0}^{R_1}h(Y_t)dt\right)^2\right], 
	\end{align*}
	and we can apply the same reasoning as in (\ref{BoundingSecondMoment}). It remains to show that the terms
	\begin{align*}
		\mathbb{E}^{(\boldsymbol{\vartheta})}_{a}\left[\int_{0}^{T_b}h(Y_t)dt\right], & & \mathbb{E}^{(\boldsymbol{\vartheta})}_{\nu}\left[\int_{0}^{T_b}h(Y_t)dt\right], & & \mathbb{E}^{(\boldsymbol{\vartheta})}_{a}\left[\left(\int_{0}^{T_b}h(Y_t)dt\right)^2\right]
	\end{align*}
	can be bounded in $\boldsymbol{\vartheta}$. The same arguments used to derive the ODEs in Appendix \ref{appendix1} can be used here to derive an ODE for $U_n(x) := \mathbb{E}^{(\boldsymbol{\vartheta})}_{x}[(\int_{0}^{T_b}h(Y_t)dt)^{n}]$ for the cases when $x<b$ and $x>b$ with the same boundary conditions as in Appendix \ref{appendix1}. Thus the following recursion holds for $U_n(x)$ when $x<b$
	\begin{align}\label{recursion1}
		U_n(x) &= n\int_{x}^{b}e^{-\int^{\xi}\frac{2\mu(\boldsymbol{\vartheta},y)}{\sigma^{2}(y)}dy}\int_{l}^{\xi}\frac{2h(\eta)}{\sigma^{2}(\eta)}e^{\int^{\eta}\frac{2\mu(\boldsymbol{\vartheta},y)}{\sigma^{2}(y)}dy}U_{n-1}(\eta)d\eta d\xi, & n&=1,2,\dots,
	\end{align}
	and for $x>b$ we have
	\begin{align}\label{recursion2}
		U_n(x) &= n\int_{b}^{x}e^{-\int^{\xi}\frac{2\mu(\boldsymbol{\vartheta},y)}{\sigma^{2}(y)}dy}\int_{\xi}^{r}\frac{2h(\eta)}{\sigma^{2}(\eta)}e^{\int^{\eta}\frac{2\mu(\boldsymbol{\vartheta},y)}{\sigma^{2}(y)}dy}U_{n-1}(\eta)d\eta d\xi,& n&=1,2,\dots.
	\end{align}
	Now for $n=1$, we get that for $x<b$,
	\begin{align*}
		\mathbb{E}^{(\boldsymbol{\vartheta})}_{x}\left[\int_{0}^{T_b}h(Y_t)dt\right] &= \int_{x}^{b}e^{-\int^{\xi}\frac{2\mu(\boldsymbol{\vartheta},y)}{\sigma^{2}(y)}dy}\int_{l}^{\xi}\frac{2h(\eta)}{\sigma^{2}(\eta)}e^{\int^{\eta}\frac{2\mu(\boldsymbol{\vartheta},y)}{\sigma^{2}(y)}dy}d\eta d\xi \end{align*}
	which is bounded over any compact set $\mathcal{K}\subset\boldsymbol{\Theta}$ by \eqref{Unbounded1}. In view of condition (\ref{Unbounded3}), we get that $\mathbb{E}^{(\boldsymbol{\vartheta})}_{\nu}[\int_{0}^{T_b}h(Y_t)dt]$ is also bounded from above in $\boldsymbol{\vartheta}$ over compact sets $\mathcal{K}\subset\boldsymbol{\Theta}$, and finally, using the recursions in (\ref{recursion1}) and (\ref{recursion2}), we get that for $x<b$,
	\begin{align*}
		\mathbb{E}^{(\boldsymbol{\vartheta})}_{x}\left[\left(\int_{0}^{T_b}h(Y_t)dt\right)^{2}\right] &= 2\int_{x}^{b}e^{-\int^{\xi}\frac{2\mu(\boldsymbol{\vartheta},y)}{\sigma^{2}(y)}dy}\int_{l}^{\xi}\frac{2h(\eta)}{\sigma^{2}(\eta)}e^{\int^{\eta}\frac{2\mu(\boldsymbol{\vartheta},y)}{\sigma^{2}(y)}dy}\mathbb{E}^{(\boldsymbol{\vartheta})}_{\eta}\left[\int_{0}^{T_b}h(Y_t)dt\right]d\eta d\xi
	\end{align*}
	which is bounded from above in $\boldsymbol{\vartheta}$ over a given compact set $\mathcal{K}\subset\boldsymbol{\Theta}$ by \eqref{Unbounded2}, giving the required bounds for the quantities $A$, $B$, and $C$. Combining these with the bounds for $D$ and $E$ we conclude that \eqref{Appendix2Aim} holds. 
\end{proof}
\bibliographystyle{abbrv}
\bibliography{bib}

\end{document}